\documentclass[leqno, a4, fleqn, 11pt]{amsart}

\usepackage{dsfont}
\usepackage{amsmath}
\usepackage{amssymb}
\usepackage{graphicx}

\usepackage{amssymb}

\usepackage{amsfonts}

\usepackage{soul}

\usepackage{mathpazo}
\usepackage{color}
\usepackage{yfonts}

\usepackage{paralist}

\usepackage{amsxtra}

\usepackage{mathtools}

\DeclarePairedDelimiter\floor{\lfloor}{\rfloor}

\def\cU{\mathcal U}

\def\cA{\mathcal A}

\newcommand{\lan}{\langle}
\newcommand{\ran}{\rangle}

\newcommand{\norm}[1]{\| #1\|}

\def\cA{          \mathcal A}
\def\cB{          \mathcal B}
\def\cC{          \mathcal C}
\def\cD{          \mathcal D}

\def\cH{          \mathcal H}

\def\clr{   \color{red}}
\def\clb{   \color{black}}

\let\cal\mathcal

\def \R{{\mathbb R}}

\def \Z{{\mathbb Z}}
\def \N{{\mathbb N}}
\def \C{{\mathbb C}}
\def \M{{\mathbb M}}
\def \H{{\mathbb H}}
\def \S{{\mathbb S}}

\newcommand{\T}{{\mathbb T}}
\newcommand{\prf}{{\begin{proof}}}
\newcommand{\epf}{{\end{proof}}}

\newcommand{\Q}{{\mathbb Q}}

\newtheorem{prop}{\sc Proposition}

\newtheorem{lemma}{\sc lemma}

\newtheorem{cor}{\sc corollary}

\theoremstyle{definition}

\def\bee{\begin{equation}}
\def\eee{\end{equation}}

\theoremstyle{rema}

\newtheorem{rema}{\sc Remark}

\newcommand{\pdvr}[2]
{\dfrac{\partial^{#2} #1}{\partial \theta^{#2_1} \partial r^{#2_2}}}
\newcommand{\pdvrs}[2]
{\partial^{#2} #1 /\partial \theta^{#2_1} \partial r^{#2_2}}
\newtheorem{thm}{Theorem}

\newtheorem{definition}{Definition}
\numberwithin{equation}{section}

\begin{document}

\setlength{\columnsep}{5pt}
\title[Ballistic Transport and Absolute Continuity]{Ballistic Transport and Absolute Continuity of One-Frequency Schr\"{o}dinger Operators}

\author{Zhiyuan Zhang, \;  Zhiyan Zhao }


\date{}
\maketitle

\tableofcontents

\begin{abstract}
For the solution $u(t)$ to the discrete Schr\"odinger equation
$${\rm i}\frac{d}{dt}u_n(t)=-(u_{n+1}(t)+u_{n-1}(t))+V(\theta + n\alpha)u_n(t), \quad n\in\Z,$$
with $\alpha\in\R\setminus\Q$ and $V\in C^\omega(\T,\R)$, we consider the
growth rate with $t$ of its diffusion norm $\langle u(t)\rangle_{p}:=\left(\sum_{n\in\Z}(n^{p}+1) |u_n(t)|^2\right)^\frac12$,
and the (non-averaged) transport exponents
$$\beta_u^{+}(p) := \limsup_{t \to \infty} \frac{2\log \langle u(t)\rangle_{p}}{p\log t}, \quad \beta_u^{-}(p):=  \liminf_{t \to \infty} \frac{2\log \langle u(t)\rangle_{p}}{p\log t}.$$
We will show that, if the corresponding Schr\"odinger operator has purely absolutely continuous spectrum, then $\beta_{u}^{\pm}(p)=1$, provided that $u(0)$ is well localized.
\end{abstract}

\section{Introduction and main result}
\noindent

For the discrete quasi-periodic Schr\"odinger operator
$$(L_\theta u)_n=-(u_{n+1}+u_{n-1})+V(\theta + n\alpha)u_n, \quad n\in\Z,$$
with $\alpha\in\R\setminus\Q$ the frequency and $V$ the potential function  on $\T:=\R/\Z$, we consider the dynamics of the equation
\begin{equation}\label{QP_Eq_Sch}
{\rm i}\frac{d}{dt}u(t)=L_{\theta}u(t).
\end{equation}
For its solution $u(t)$, we want to observe the growth rate with $t$ of the ``diffusion norm"(or in some context also known as ``2nd-moment of the position"):
$$\langle u(t)\rangle_{2}:=\left(\sum_{n\in\Z}(n^2+1) |u_n(t)|^2\right)^\frac12.$$
More generally, we can define the $p^{\rm th}-$moment of $u(t)$ for any $p \geq 0$ by
 $$\langle u(t)\rangle_{p}:=\left(\sum_{n\in\Z}(|n|^{p}+1) |u_n(t)|^2\right)^\frac12.$$
We define the subspace
${\cal W}^{p}(\Z) := \{ u \in \ell^{2}(\Z) : \langle u \rangle_p < \infty \}$.

It is known that when the initial condition $u(0)$ is well-localised, we have $\langle u(t)\rangle_{p}<\infty$ for any finite $t$ and $p \geq 0$ (see Theorem 2.22 of \cite{DamanikTch}).
One standard way to describe the propagation of $u(t)$ in space is to consider the asymptotic growth of the $p^{\rm th}-$moment norm. This is stated in terms of transport exponents
$$\beta^{+}_{u}(p) := \limsup_{t \to \infty} \frac{2\log \langle u(t)\rangle_{p}}{p\log t}, \quad \beta^{-}_{u}(p):=  \liminf_{t \to \infty} \frac{2\log \langle u(t)\rangle_{p}}{p\log t}.$$

There is an extensive study of transport exponents and the time-averaged variants. It is understood that the transport behaviour is intimately related to the spectral properties of the operator. In the case where $L_\theta$ has only pure point spectrum, Simon\cite{Simon1990} showed that  for compactly supported $u(0)$,
$$
\lim_{t \to \infty} t^{-1} \langle u(t) \rangle_{2} = 0.
$$
On the other hand, the absolutely continuous(from now on, a.c. for short) part of the spectrum is expected to be associated with the strongest transport property, which is usually called ballistic transport.
More precisely, for the solution $u(t)$ to Eq. (\ref{QP_Eq_Sch}), it is expected that the norm $\langle u(t)\rangle_{p}$ grows like  $t^{p/2}$.
 As a general result in this direction, there is a time-averaged statement by Guarneri-Combes-Last theorem\cite{Last}, which shows that, in the presence of a.c. spectrum, for any well-localised $u \neq 0$ in the subspace of $\ell^2(\Z)$ corresponding to the a.c. spectrum, there exists some positive constant $C$ such that
$$
\liminf_{t\to\infty}\frac1T\int_0^T \langle u(t)\rangle_{p}^2 \, dt\geq C T^{p}.
$$
There is also a large body of literature devoted to the study of transport property for operators with singular continuous spectrum. We will refer the readers to survey \cite{DamanikTch} and introductions in \cite{DamanikTch2} for details. For more descriptions of the diffusion norm of nonlinear operators, we refer to \cite{BW}.

The main subject of this paper is to investigate the transport property of operators with a.c. spectrum. It is fairly natural to ask whether one can go beyond the averaged version in Guarneri-Combes-Last theorem. It turns out that this is not a simple generalisation, possibly due to lack of good spectral quantity associated to terms like $\langle u(t) \rangle_{p}$.  Recently, Damanik-Lukic-Yessen\cite{DamanikLY} have shown the stronger version of ballistic motion for $p=2$(i.e., the above inequality without time-averaging) for the periodic Schr\"odinger equation.
This is an extension of the work of Asch-Knauf\cite{AschKnauf} for Schr\"{o}dinger operators.
Zhao\cite{Zhao} has proven it for the quasi-periodic ones with small potential and Diophantine frequencies. Both results correspond to the Schr\"odinger operator with purely a.c. spectrum.

In this paper, we establish a complete link between purely a.c. spectrum and ballistic transport for (analytic) one-frequency Schr\"odinger operator.
The main result is stated as follows.

%
\begin{thm} \label{main theorem 2}
 Given $\alpha \in \R \setminus \Q$, and $V \in C^{\omega}(\T, \R)$ such that the Schr\"{o}dinger operator $L=L_{\theta}$ has purely a.c. spectrum for a.e. $\theta$. Given any $\eta > 0$ and $p\geq 0$.
 \begin{itemize}
   \item [1)] Consider the solution $u(t)$ of Eq. (\ref{QP_Eq_Sch}) with $u(0) \in {\cal W}^{p'}(\Z) \setminus \{0\}$, where $p'=p$ if $p \in 2\Z$ and $p'= 2\lfloor\frac{p}{2}\rfloor +2$ otherwise. For every $\theta\in\T$, we have
$$\lim_{t \to \infty}\frac{\langle u(t)\rangle_{p}^2}{t^{p+\eta}}=0.$$
   \item [2)] If $p \geq 2$ or $p=0$ and $u(0) \in \cal{W}^{p}(\Z) \setminus \{0\}$, or if $0 < p < 2$ and $u(0) \in {\cal W}^{4}(\Z) \setminus \{0\}$, then for a.e. $\theta\in\T$, we have
$$ \lim_{t \to \infty}\frac{\langle u(t)\rangle_{p}^2}{t^{p-\eta}} = \infty.$$
 \end{itemize}
 In particular, for a.e. $\theta\in\T$, for any exponentially decaying $u(0)\neq0$ , we have $\beta^{+}_u(p) = \beta^{-}_u(p) = 1$.
\end{thm}

\begin{rema}
In view of the work of Damanik-Lukic-Yessen for the periodic potential case (Theorem 1.6 of \cite{DamanikLY}), the conclusion of Theorem \ref{main theorem 2} holds for any $\alpha\in\R$.
\end{rema}


\begin{rema}
According to Corollary 1.7 of \cite{A}, the conclusion of \ref{main theorem 2} holds if $V$ is close to constant.
\end{rema}

Besides the intrinsic interest in this problem, another motivation can be found in connection with the so-called XY spin chain, studied in many-body quantum physics. In \cite{DamanikLY}, the authors established lower bound for the Lieb-Robinson velocity for the anisotropic XY chain on $\Z$ with periodic parameters as an application of their proof of the ballistic motions. Also through this connection with Schr\"{o}dinger operators, Kachkovskiy\cite{Kachkovskiy} has proven similar lower bound for a class of isotropic quasi-periodic XY spin, corresponding to analytic cocycles which are reducible for almost every energy, including the small analytic regime in \cite{Zhao}, and the cases with purely a.c. spectrum and a single Diophantine frequency. It is explained in \cite{Kachkovskiy} that although the transport property in two models share some connections, the averaged lower bound of Guarneri-Combes-Last does not translate into useful informations for XY spin chains. We hope that the method in this paper could provide some information on how to approach XY spin chains.

\

\noindent{\bf Idea of Proof.}
Since the general ballistic upper bound is known(see Theorem \ref{ballistic_upper_general} in Section \ref{growth}), we only need to show the lower bound.

Following the main strategy of \cite{Zhao}, we relate the growth of the diffusion norm to the so-called ``modified spectral transformation".
Roughly speaking, one comes down to show that the Bloch-wave at different sites has decaying correlation with respect to some well-chosen measure.
In \cite{Zhao}, a natural candidate measure for this construction is the measure defined by the integrated density of states.
The derivative of the Floquet exponent serves as the source of the decay.
While in our case, two difficulties arises.
\begin{itemize}
  \item The first is that the usual KAM breaks down, namely in general one can no longer reduce the Schr\"odinger cocycle to constant.
Avila-Fayad-Krikorian\cite{AFK} developed a theory which allows one to reduce the cocycle to a (phase-dependent) rotation.
Through this type of reducibility, we can construct the ``generalized Bloch-wave" with ``phase-dependent Floquent exponent".
The phase-dependence complicated the matter(see Subsection 5.4).
  \item The second difficulty is that we need to exploit the second order derivatives of the modified spectral measure, while the rotation number only has good first order derivative estimates. This difficulty does not appear in the construction of \cite{Zhao}, in which the Floquent exponent is just the rotation number and the form of Bloch-wave is much simpler, due to the reducibility to constant. Then an integration by parts can be performed before exploiting the regularity of the rotation number. While in our case we can not expect to find a parametrisation to accommodate the variation of phase-dependent Floquent exponent at different sites. So we need to make a non-canonical choice of measure that comes in the construction of the modified spectral transformation.
To retain the decay of the correlations, one has to study the regularity of the measure in question.
Since one often expects the spectrum to be a Cantor set, we can not expect to choose a measure with summable sequence of Fourier coefficients. This is why we are only able to study the transport exponents but not the linear lower bound as in \cite{Zhao}.
\end{itemize}

To get the ballistic lower exponent, our main observation is: by the renormalization theory developed by Avila-Krikorian(\cite{AK} and its generalisation \cite{AK2}), we can initialize the KAM for arbitrarily fine data, which allows us to delay the occurrence of resonance at will.
We will construct a measure supported on spectrum adapted to the KAM scheme, with good regularity at finite yet arbitrarily long intervals of scales. Combining with a truncated version of modified spectral transform, we can complete the proof.


%
%
%

\section{Preliminaries and notations}

\subsection{Schr\"{o}dinger operator and Schr\"{o}dinger cocycle}\label{pre_op_cocycle}

We recall some basic notions and well-known results for the quasi-periodic Schr\"{o}dinger operator
  $L=L_\theta:\ell^2(\Z)\rightarrow \ell^2(\Z)$,
$$
(L u)_n=-(u_{n+1}+u_{n-1})+V(\theta+n\alpha) u_n,
$$
with $\alpha\in\R\setminus\Q$, $V\in C^\omega(\T,\R)$ and the corresponding Schr\"{o}dinger cocycle $(\alpha, A_{(E,V)})$:
\begin{equation}\label{qpcocycle}
\left(\begin{array}{c}
u_{n+1} \\[1mm]
u_{n}
\end{array}
\right)=A_{(E,V)}(\theta+n\alpha)\left(\begin{array}{c}
u_{n} \\[1mm]
u_{n-1}
\end{array}
\right)  \  {\rm with}  \  A_{(E,V)}(\theta):=\left(\begin{array}{cc}
            -E+V(\theta) & -1 \\[1mm]
            1 & 0
          \end{array}
\right).
 \end{equation}
Note that $(\alpha, A_{(E,V)})$ is equivalent to the eigenvalue problem $Lu=Eu$.

\subsubsection{Spectral measure and integrated density of states}
Let $\sigma(L)$ denote the spectrum of $L$.
Fixing any phase $\theta\in\T$ and any $\psi\in\ell^2(\Z)$, let $\mu_\theta=\mu_{\theta,\psi}$ be
the spectral measure of $L = L_\theta$ corresponding to $\psi$, which is defined so that
$$\lan(L_\theta-E)^{-1}\psi, \psi \ran = \int_\R\frac{1}{E-E'}d\mu_{\theta,\psi}(E'), \quad \forall \, E\in \C\setminus\sigma(L).$$
From now on, we restrict our consideration to $\mu_\theta=\mu_{\theta, \delta_{-1}}+\mu_{\theta, \delta_0}$ and just call it the {\bf spectral measure}, where
$\{\delta_n\}_{n\in\Z}$ is the canonical basis of $\ell^2(\Z)$.
Since $\{\delta_{-1}, \delta_0\}$ forms a generating basis of $\ell^2(\Z)$\cite{CarmonaLacroix},
that is, there is no
proper subset of $\ell^2(\Z)$ which is invariant by $L$ and contains $\{\delta_{-1}, \delta_0\}$. In particular
the support of $\mu_\theta$ is $\sigma(L)$ and if $\mu_{\theta}$ is a.c. then any $\mu_{\theta,\psi}$ , $\psi\in\ell^2(\Z)$, is a.c. .

The {\bf integrated density of states} is the function $k:\R\rightarrow[0,1]$ such that
$$k(E)=\int_{\T} \mu_\theta(-\infty,E] \,  d\theta,$$
which is a continuous non-decreasing surjective function.

\subsubsection{Rotation number and Lyapunov exponent}
Related to the Schr\"odinger cocycle, a unique representation can be given for the rotation number
$\rho=\rho_{(\alpha, A_{(E,V)})}$.
Indeed, the rotation number is defined for more general quasi-periodic cocycles.
It is introduced originally by Herman\cite{Herman} in this
discrete case(see also Delyon-Souillard\cite{DelyonSouillard}, Johnson-Moser\cite{JM}, Krikorian\cite{Kri}).
For the precise definition, we follow the same presentation as in \cite{AFK}.

Given $A(\cdot):T \rightarrow SL(2,\R)$, continuous and homotopic to the identity, the same is true for
$$
\begin{array}{llll}
F:& \displaystyle  \T\times \S^1 &\rightarrow& \displaystyle  \T\times \S^1 \\[3mm]
   &(\theta,v)&\mapsto&(\theta+\alpha, \, \frac{A(\theta)v}{\|A(\theta)v\|})
\end{array}.$$
Therefore, $F$ admits a continuous lift $\tilde F:\T\times\R\rightarrow \T\times\R$ of the form $\tilde F(\theta, x)=(\theta + \alpha, x + f(\theta, x))$ such that $f(\theta, x + 1) = f(\theta, x)$ and $\Pi(x+f(\theta, x))=\frac{A(\theta)\Pi(x)}{\|A(\theta)\Pi(x)\|}$, where $\Pi:\R\rightarrow S^1$, $\Pi(x)=e^{{\rm i}2\pi x}:=(\cos2\pi x, \sin2\pi x)$.
In order to simplify the terminology, we can say that $\tilde F$ is a lift for $(\alpha,A)$.
The map $f$ is independent of the choice of the lift up to the addition of a constant integer
$p\in \Z$. Following \cite{Herman} and \cite{JM}, we define the limit
$$\lim_{n\rightarrow\pm\infty}\frac{1}{n}\sum_{k=0}^{n-1}f(\tilde F^{k}(\theta,x)),$$
which is independent of $(\theta, x)$ and where the convergence is uniform in $(\theta, x)$.
The class of this number in $\T$, which is independent of the chosen lift, is called the {\bf fibered rotation number} of $(\alpha,A)$ and denoted by $\rho_{(\alpha,A)}$.
Moreover, $\rho_{(\alpha,A)}$ is continuous as
a function of $A$(with respect to the uniform topology on $C^0(\T, SL(2,R)$), naturally
restricted to the subset of $A$ homotopic to the identity).

\

For the quasi-periodic cocycle
$\left(\begin{array}{c}
u_{n+1} \\
u_{n}
\end{array}
\right)=A(\theta+n\alpha)\left(\begin{array}{c}
u_{n} \\
u_{n-1}
\end{array}
\right)$
with $\alpha\in\R\setminus\Q$,
the {\bf Lyapunov exponent} $\gamma=\gamma_{(\alpha,A)}$ is defined by
$$\gamma_{(\alpha,\, A)}:=\lim_{n\rightarrow \infty}\frac{1}{n}\int_{\T}\ln|A(\theta+n\alpha)\cdots A(\theta+\alpha)|\, d\theta.$$
By Kingman's subadditive ergodic theorem,
$$\gamma_{(\alpha,\, A)}:=\lim_{n\rightarrow \infty}\frac{1}{n}\ln|A(\theta+n\omega)\cdots A(\theta+\omega)|.$$

In particular, for quasi-periodic Schr\"{o}dinger cocycle $(\alpha, A_{(E,V)})$ given in (\ref{qpcocycle}), a well-known result of Kotani theory shows, if the linear Schr\"{o}dinger operator $L$ has purely absolutely continuous spectrum, then $\gamma(E)=0$ a.e. on $\sigma(L)$.
Moreover,
the Thouless formula relates the Lyapunov exponent to the integrated density of states:
$$\gamma(E)=\gamma_{(\alpha, A_{(E,V)})} = \int_{\R} \ln|E'-E| \, dk(E').$$
There is also a relation between the rotation number and  the integrated density of states:
$$k(E)=\left\{ \begin{array}{cl}
                 0, & E\leq \inf\sigma(L) \\[1mm]
                \frac{\rho(E)}{\pi}, & \inf\sigma(L)<E<\sup\sigma(L) \\[1mm]
                 1, & E \geq \sup\sigma(L)
               \end{array}
 \right.  .$$
By gap-labelling theorem(see, e.g., \cite{DelyonSouillard, JM}), $k(E)=\frac{\rho(E)}{\pi}$ is constant in a gap of $\sigma(L)$(i.e., an interval in the resolvent set of $L$), and each gap is labelled with $l\in\Z$ such that $\rho=\frac{ l\alpha}{2}$ mod $\pi$ in this gap.

\

\subsubsection{The $m-$functions}

The spectral measure $\mu=\mu_\theta$ can be studied through
its Borel transform $M = M_\theta$:
$$M(z)=\int\frac{1}{E'-z}d\mu(E').$$
It maps the upper-half plane $\H:=\{z\in\C: \Im z>0\}$ into itself.

From the limit-point theory, for $z\in\H$, there are two solutions $u^{\pm}$, with $u_0^{\pm}\neq 0$, which are $\ell^2$ at $\pm\infty$ and satisfying $Lu^{\pm} = zu^{\pm}$,
defined up to normalization.
Let
$m^{\pm}:=-\frac{u^{\pm}_{\pm1}}{u^{\pm}_{0}}$.
$m^+$ and $m^-$ are Herglotz functions, i.e., they map $\H$ holomorphically into itself(see, e.g., \cite{Simon1983} for more properties of Herglotz function).
Moreover, it is well known that
$$M =\frac{m^+m^- -1}{m^+ + m^-}.$$
By the property of Herglotz function, we know that for almost every $E\in\R$,
the non-tangential limits $\lim_{\epsilon\rightarrow0}m^{\pm}(E+{\rm i}\epsilon)$ exist, and they
define measurable functions on $\R$ which we still denote $m^\pm(E)$.

We have the following key result of Kotani Theory\cite{Simon1983}.
\begin{lemma}[Theorem 2.2 of \cite{Avila}]
For every $\theta$, for a.e. $E$ such that $\gamma(E) = 0$, we have
$m^+(E)=m^-(E)$.
\end{lemma}

\subsubsection{Classical spectral transformation}

%

Let $u(E)$ and $v(E)$ be the solutions of the eigenvalue problem $Lq=Eq$ such that
$\left(
\begin{array}{cc}
u_1 & v_1 \\
u_0 & v_0 \\
\end{array}
\right)=\left(
\begin{array}{cc}
1 & 0 \\
0 & 1 \\
\end{array}
\right)
$.
We have
\begin{thm}[Chapter 9 of \cite{CL}]\label{spectral_measure_matrix}
There exists a non-decreasing Hermitian matrix $\mu=(\mu_{jk})_{j,k=1,2}$ whose elements are of bounded variation on every finite interval on $\R$, satisfying
$$\mu_{jk}(E_2)-\mu_{jk}(E_1)=\lim_{\epsilon\rightarrow 0_+}\frac{1}{\pi}\int_{E_1}^{E_2}\Im M_{jk}(\nu+i\epsilon)d\nu,$$
at points of continuity $E_1$, $E_2$ of $\mu_{jk}$, where on $\H$,
$$M=\left(\begin{array}{cc}
            M_{11} & M_{12} \\
            M_{21} & M_{22}
          \end{array}
\right):=-\frac{1}{m^{+}+m^{-}}\left(\begin{array}{cc}
                                  1 & m^+ \\
                                  -m^- & -m^+m^-
                                \end{array}
\right),$$
such that for any $q\in \ell^2(\Z)$, with
$(g_1(E),g_2(E)):=\left(\sum_{n\in\Z} q_n u_n(E), \sum_{n\in\Z} q_n v_n(E)\right)$,
we have Parseval's equality
$$\sum_{n\in\Z}|q_n|^2=\int_{\R}\sum_{j,k=1}^2 \bar g_j(E) g_k(E) d\mu_{jk}(E).$$
\end{thm}

Given any matrix of measures on $\R$ $d\varphi=\left(\begin{array}{cc}
                 d\varphi_{11} & d\varphi_{12} \\[1mm]
                 d\varphi_{21} & d\varphi_{22}
               \end{array}\right)$,
let ${\cal L}^2(d\varphi)$ be the space of vectors
$G=(g_j)_{j=1,2}$, with $g_j$ functions of $E\in\R$ satisfying
\begin{equation}\label{gen_L2}
\|G\|_{{\cal L}^2(d\varphi)}^2:=\sum_{j,k=1}^2 \int_\R g_j \, \bar g_k \,  d\varphi_{jk}<\infty.
\end{equation}
In view of Theorem \ref{spectral_measure_matrix}, the map
$(q_n)_{n\in\Z}\mapsto \left(\begin{array}{c}
                                                         \sum_{n\in\Z}q_n u_n(E) \\[1mm]
                                                         \sum_{n\in\Z}q_n v_n(E)
                                                       \end{array}
     \right)$
defines a unitary transformation between $\ell^2(\Z)$ and ${\cal L}^2(d\mu)$.
We call it as the {\bf classical spectral transformation}.

By Chapter $\uppercase\expandafter{\romannumeral5}$ of \cite{PF}(Page 297), we know that
the matrix of measures $(d\mu_{jk})_{j,k=1,2}$ is Hermitian-positive, and
therefore each $d\mu_{jk}$ is a.c. with respect to the measure $d\mu_{11}+d\mu_{22}$.
This measure is a.c. with respect to the above spectral measure $\mu_\theta=\mu_{\theta, e_{-1}}+\mu_{\theta, e_0}$
and it determines the spectral type of the operator.
In particular, if the spectrum of $L$ is purely a.c., we have, for any $q\in\ell^2(\Z)\setminus\{0\}$,
the classical spectral transformation is supported on a subset of $\sigma(L)$ with positive Lebesgue measure.

\

For the classical spectral transformation, there are some singularities with respect to $E$. More precisely, $u_n$ and $v_n$ are not well differentiated somewhere in the spectrum $\sigma(L)$.
For example, for the free Schr\"{o}dinger operator
$(L q)_n=-(q_{n+1}+q_{n-1})$,
we have $\sigma(L)=[-2,2]$ and for $E\in\sigma(L)$ the rotation number is $$\xi_0(E)=\rho_{(\omega, A_{(E,0)})}(E)=\cos^{-1}\left(-\frac{E}{2}\right)\in[0,\pi].$$
Since $-E=2\cos\xi_0$, we can see that the two generalized eigenvectors
\begin{equation}\label{classical_uv}
u_n=\frac{\sin n\xi_0}{\sin \xi_0}, \quad v_n=-\frac{\sin (n-1)\xi_0}{\sin \xi_0}
\end{equation}
satisfy $\left(
\begin{array}{cc}
u_1 & v_1 \\
u_0 & v_0 \\
\end{array}
\right)=\left(
\begin{array}{cc}
1 & 0 \\
0 & 1 \\
\end{array}
\right)
$
and, on $(-2,2)$, $\xi_0'=\frac{1}{2\sin\xi_0}$.
Differentiating $u_n$, we have
$$u'_n=\frac{1}{2\sin\xi_0}\left( \frac{n\cos n\xi_0}{\sin \xi_0}-\frac{\sin n\xi_0\cdot\cos\xi_0}{\sin^2 \xi_0} \right).$$
The singularity comes when $\xi_0$ approaches $0$ and $\pi$.

\subsection{Denominators of continued fraction expansion}
Define as usual for $0 <\alpha< 1$,
$$a_0=0,\quad \alpha_0=\alpha,$$
and inductively for $k\geq 1$, as the fractional part of
$$ a_k:=\max\{n\in\Z: n\leq \alpha_{k-1}^{-1}\}, \quad \alpha_k:=\alpha_{k-1}^{-1}-a_k. $$
Then we define
$$p_0=0, \;\ q_0=1,  \quad p_1=a_1, \;\ q_1=1,$$
and inductively,
$\left\{\begin{array}{c}
           p_k=a_k p_{k-1}+p_{k-2} \\[1mm]
           q_k=a_k q_{k-1}+q_{k-1}
         \end{array}
\right. $.

Recall that the sequence $(q_n)$ is the sequence of best denominators of $\alpha\in \R\setminus\Q$
since it satisfies
$$\|k\alpha\|_{\T}\geq\|q_{n-1}\alpha\|_{\T}, \quad \forall \, 1\leq k\leq q_n,$$
and $\|q_n \alpha\|_\T \leq \frac{1}{q_{n+1}}$, where $\|x\|_{\T}:=\inf_{p\in\Z}|x-p|$ for $x\in\R$.

\subsection{Regularity in the sense of Whitney}

Given a closed subset $S$ of $\R$ and $r\in\Z_+$. We give a precise definition of $C^r$ in the sense of Whitney, corresponding to a more general definition in \cite{Poschel}.

\begin{definition}[$C^r$ in the sense of Whitney]\label{definition_whitney}
Given $r+1$ functions $F_k:S\rightarrow \C$(or $\mathbb{M}(2,\C)$), $k=0,\cdots, r$, and some $0<M<\infty$, such that for $k=0,\cdots, r$,
\begin{equation}\label{condition_Whitney}
|F_k(x)| \leq M, \;\  |F_k(x)-P_k(x,y)|\leq M|x-y|^{1-\frac{k}{r+1}}, \quad  \forall \, x,y\in S,
\end{equation}
where $P_k(x,y):=\sum_{k+l\leq r}\frac{1}{l!}F_{k+l}(y)(x-y)^l$.
We say that $F_0$ is {$C^r$ in the sense of Whitney} on $S$, denoted by $F_0\in C_W^r(S)$, with the $k^{\rm th}-$order derivative $F_k$, $k=1,\cdots,r$.
The $C_W^r(S)-$norm of $F_0$ is defined as
$$|F_0|_{C_W^r(S)}:=\inf M.$$
\end{definition}

\begin{rema}
By Whitney's extension theorem\cite{Whitney}, we can find an extension $\tilde F:\R\rightarrow \C$, which is $C^r$ on $\R$ in the natural sense, such that
$\tilde F|_S=F_0$ and $\tilde F^{(k)}|_S=F_k$. Indeed, the estimation for the upper bound of the $C_W^r(S)-$norm can be realized by estimating the extension.
\end{rema}


\subsection{Notations.}
{\rm 1)} For any $x\in\R$, let $\lfloor x \rfloor:=\max\{n\in\Z: n\leq x\}$ and $\|x\|:=\|x\|_{\T}=\inf_{p\in\Z}|x-p|$.

\smallskip

\noindent {\rm 2)} For any subset $S\subset\R$, let $|S|$ be its Lebesgue measure, $cl\{S\}$ be its closure and
for any $r>0$, $B(S,r):=\{x\in\R: |x-y|<r, \exists\,  y\in S\}$. Given any function (possibly matrix-valued) $f(E)$ on $\R$, let $supp(f)$ be its support, and $\partial_E^k$ be its $k^{\rm th}-$order derivative(possibly in the sense of Whitney).

\smallskip

\noindent {\rm 3)} For $h>0$, let $\T_h:=\{z\in\C / \Z:   |{\Im} z|<h\}$. Given any analytic function $f$ on $\T_h$, let $\hat f(n)$ be its $n^{\rm th}-$Fourier coefficient, $n\in\Z$.

\smallskip

\noindent {\rm 4)}  Given $\alpha\in\R$, for any function $\phi : \T_h \to \C$ and any $n \in \N$, we define the Birkhoff sum of $\phi$ over $\theta \mapsto\theta + \alpha$ by
$$ \phi^{[n]}(x) := \sum_{k=0}^{n-1}\phi(\theta + k\alpha), \quad \phi^{[-n]}(x) := \sum_{k=-n}^{-1}\phi(\theta + k\alpha).$$
Moreover, we denote $R_{\phi}:=\left(\begin{array}{cc}
                   \cos\phi & -\sin\phi \\
                   \sin\phi & \cos\phi
                 \end{array}\right)$.

\smallskip

\noindent{\rm 5)} Given a compact set $S\subset\R$,
for any function $F$ on $S \times \T_h$, denote
$$|F|_{S,\, h}:=\sup_{E\in S}\sup_{\theta\in\T_h}|F(E,\theta)|.$$
When $S$ is a compact set or a finite union of intervals, denote $$|F|_{C^j(S),\, h}:=\sup_{E\in S}\sup_{\theta\in\T_h}\left(\sum_{l=0}^j|\partial_{E}^{l}F(E,\theta)|\right), \  j\in\Z_+.$$
For the function $F$ defined on $S\times \T$, we define the norms $|F|_{S}$ and $|F|_{C^j(S)}$ in a similar fashion, without showing the subscriprt ``$\T$" explicitly. If there is no confusion, we denote $|\cdot|_{S}$ by $|\cdot|$.

\smallskip

In this paper, in the formulations and proofs of various assertions,
we shall encounter constants which maybe depend on various quantities, but independent of the index $l$ which represents the iteration step,
 or the variables (i.e., the phase $\theta$, the energy $E$ and the time $t$). All such constants
will be denoted by $c$, $c_1$, $c_2$, $\cdots$, and sometimes even
different constants will be denoted by the same symbol if there is no ambiguity.
Moreover, in some estimates, we use the notation ``$\lesssim$" or ``$\gtrsim$" instead of showing the numerical constant explicitly.

\section{Admissible subsequences of denominators}

In this section, we will explain how to choose subsequences of denominators of an irrational number, which will later serve as indices of resonance in the KAM scheme.
As an application, we construct two subsequences depending on a given parameter $\eta>0$, so that for any $T$ sufficiently large, the interval of the form $[T^{\eta}, T]$ will be well contained in one interval defined by these subsequences(see the definitions of ${\cal M}_l$ and $\tilde{\cal M}_l$ in Subsection 5.5 for detail).

Given $\alpha\in\R\setminus\Q$, with $(q_n)$ its sequence of best denominators.

\begin{definition}[$CD({\cA}, {\cB}, {\cC})$ bridge, Definition 3.1 in \cite{AFK}]
Given $0 < {\cA} \leq {\cB} \leq {\cC}$. For integers $1 \leq l < n$,
 we say that the pair of denominators $(q_l, q_n)$ forms a $CD({\cA}, {\cB}, {\cC})$ bridge if
\begin{itemize}
\item $q_{i+1} \leq q_i^{{\cA}}$, $\forall \, i = l, \cdots, n-1$;
\item $q_{l}^{{\cC}} \geq q_n \geq q_{l}^{{\cB}}$.
\end{itemize}
\end{definition}

Given $0<{\cA}\leq {\cB}'\leq{\cB}\leq{\cC}\leq{\cC}'$, it is obvious that any $CD({\cA}, {\cB}, {\cC})$ bridge is also a $CD({\cA}, {\cB}', {\cC}')$ bridge.

\smallskip

For any subsequence $(q_{n_k})$ of $(q_n)$,
we follow the notations in \cite{AFK} and denote subsequences $(Q_{k}) = (q_{n_k})$ and $(\overline{Q}_k)=(q_{n_k + 1})$.
The properties of these notations have been exploited in Proposition 3.1 in \cite{AFK}. We will recall it later.
\begin{definition}[$({\cA}, {\cB}, {\cC}, {\cD})-$admissible] Given $0 < {\cA} \leq {\cB} \leq {\cC}$ and ${\cD}>0$,
a subsequence $(Q_k)$ is called $({\cA}, {\cB}, {\cC}, {\cD})-$admissible if for any $k \geq 1$
\begin{itemize}
\item $Q_{k} \leq \overline{Q}_{k-1}^{{\cD}}$;
\item either $\overline{Q}_k > Q_k^{{\cA}}$ or $(\overline{Q}_{k-1}, Q_{k})$, $(Q_{k}, Q_{k+1})$ are both $CD({\cA}, {\cB}, {\cC})$ bridges.
\end{itemize}
\end{definition}

Given $0<{\cA}\leq {\cB}'\leq{\cB}\leq{\cC}\leq{\cC}'$ and $0<{\cD}\leq{\cD}'$, it is obvious that any $({\cA}, {\cB}, {\cC}, {\cD})-$admissible subsequence is also $({\cA},{\cB}', {\cC}', {\cD}')-$admissible.

\clr

\clb

\begin{lemma}[essentially Lemma 3.2 of \cite{AFK}]\label{lemma_Qk}
Given any $\mathcal{A} > 1$, there exists an $(\cA, \cA^{3}, \cA^{21}, \cA^{20})-$admissible subsequence $(Q_k)$ with $Q_0 =1$.
\end{lemma}
\begin{rema}
In Lemma 3.2 of \cite{AFK}, the authors showed the existence of $(\cA, \cA, \cA^3, \cA^4)-$admissible subsequences. Lemma \ref{lemma_Qk} can be shown following their proof with obvious modifications.
\end{rema}

\smallskip

In particular, from the definition of $(\cA, \cA^{3}, \cA^{21}, \cA^{20})-$admissible subsequence, we can easily see that $Q_{k+1}\geq Q_{k}^{\cA}$ for each $k$. Indeed,
\begin{itemize}
  \item If $\overline{Q}_{k} > Q_{k}^{\cA}$, it is obvious since $\overline{Q}_k\leq Q_{k+1}$.
  \item If $\overline{Q}_{k} \leq Q_{k}^{\cA}$, we have $(Q_{k}, Q_{k+1})$ is a $CD(\cA, \cA^{3}, \cA^{21})$ bridge, so $Q_{k+1} \geq Q_{k}^{\cA^{3}} \geq  Q_{k}^{\cA}$.
\end{itemize}

\begin{lemma} \label{lemma two sequences}
Let $\mathcal{A}>1$ and $(Q_k)$ be an $(\cA, \cA^{3}, \cA^{21}, \cA^{20})-$admissible subsequence. There exists a subsequence $(R_k)$ such that $R_0 =1$, and
\begin{itemize}
  \item [\rm (1)]$Q_{k}^{\cA}\leq R_{k}^{\cA}\leq Q_{k+1}$ and $\overline{R}_k \geq Q_{k}^{\cA}$ for each $k \geq 1$;
  \item [\rm (2)]$(R_k)$ is $(\cA, \cA, \cA^{22},\cA^{21})-$admissible.
\end{itemize}
\end{lemma}

\proof For $k \geq 1$, define $R_{k}$ as the largest denominator such that $Q_{k} \leq R_{k} < Q_{k}^{\cA}$. Then we have $\overline{R}_k \geq Q_{k}^{\cA}$ since $\overline{R}_k > R_k $. We have two cases :
\begin{itemize}
  \item If $\overline{Q}_{k} > Q_{k}^{\cA}$, we can see $R_{k} = Q_{k}$. So $R_{k}^{\cA}\leq \overline{Q}_{k} (=\overline{R}_{k})\leq Q_{k+1}$.
  \item If $\overline{Q}_{k} \leq Q_{k}^{\cA}$, we have that $Q_{k+1} \geq Q_{k}^{\cA^{3}} \geq R_{k}^{\cA^2}$.
\end{itemize}
In both case we have $Q_{k+1} \geq R_{k}^{\cA}$. (1) is proven.

It remains to verify that $(R_k)$ is $(\cA, \cA, \cA^{22},\cA^{21})-$admissible.  By the property of $(Q_k)$ and the definition of $(R_k)$, we have
\begin{equation}\label{CD_bridge2}
R_{k} \leq Q_{k}^{\cA} \leq \overline{Q}_{k-1}^{\cA^{21}} \leq \overline{R}_{k-1}^{\cA^{21}} \ {\rm for \ every } \ k\geq1.
\end{equation}
Assume that $\overline{R}_{k}\leq R_{k}^{\cA}$ for some $k$(otherwise we can finish the proof without showing that $(\overline{R}_{k-1}, R_{k})$, $(R_{k}, R_{k+1})$ are both $CD({\cA}, {\cA}, {\cA}^{22})$ bridges).
As shown above, $\overline{Q}_{k} > Q_{k}^{\cA}$ implies $\overline{R}_k> R_{k}^{\cA}$, so we have $\overline{Q}_{k} \leq Q_{k}^{\cA}$.
Thus $(\overline{Q}_{k-1}, Q_{k})$, $(Q_{k}, Q_{k+1})$ are both $CD({\cA}, {\cA}^3, {\cA}^{21})$ bridges.

 By (1), we have $\overline{Q}_{k-1} \leq \overline{R}_{k-1} \leq Q_{k}\leq R_k$. Since  $(\overline{Q}_{k-1}, Q_{k})$ is a $CD({\cA}, {\cA}^3, {\cA}^{21})$ bridge, we have that $R_{k} \geq Q_{k} \geq \overline{Q}_{k-1}^{\cA^{3}}$ and
 \begin{equation}\label{CD_bridge111}
q_{i+1}\leq q_i^{{\cal A}} \quad  {\rm for \ every \ } \overline{Q}_{k-1}\leq q_i < Q_k.
 \end{equation}
 Moreover, since $R_{k} < Q_k^{\cA}$, we know that
 \begin{equation}\label{CD_bridge112}
q_{i+1}\leq q_i^{{\cal A}} \quad  {\rm for \ every \ }Q_{k}\leq q_i < R_k.
 \end{equation}
If $\overline{Q}_{k-1} > Q_{k-1}^{\cA}$, then $R_{k-1} = Q_{k-1}$ and as a result $\overline{R}_{k-1} = \overline{Q}_{k-1}$. Otherwise we have that
$\overline{R}_{k-1} \leq R_{k-1}^{\cA} < Q_{k-1}^{\cA^{2}}$.
In both case we have $\overline{R}_{k-1} \leq \overline{Q}_{k-1}^{\cA^{2}}$.
Thus $\overline{R}_{k-1}^{\cA}\leq \overline{Q}_{k-1}^{\cA^{3}} \leq R_{k}$.
Combing with (\ref{CD_bridge2}), (\ref{CD_bridge111}) and (\ref{CD_bridge112}), we get that $(\overline{R}_{k-1}, R_{k})$ is a $CD(\cA, \cA, \cA^{22})$ bridge.

By the property of the $CD(\cA, \cA^{3}, \cA^{21})$ bridge $(Q_k, Q_{k+1})$, together with the definition of $(R_k)$,
 we have $R_{k}^{\cA^{2}} \leq Q_{k}^{\cA^3} \leq Q_{k+1} \leq R_{k+1}\leq Q_{k+1}^{\cA} \leq Q_{k}^{{\cA}^{22}} \leq R_{k}^{{\cA}^{22}}$ and
 $$q_{i+1}\leq q_i^{{\cal A}} \quad  {\rm for \ every \ } {R}_k\leq q_i < Q_{k+1}.$$
Since $R_{k+1} \leq Q_{k+1}^{\cA}$, we also have
$$q_{i+1}\leq q_i^{{\cal A}} \quad  {\rm for \ every \ }Q_{k+1}\leq q_i < R_{k+1}.$$
Thus $(R_k, R_{k+1})$ is a $CD(\cA, \cA, \cA^{22})$ bridge. This completes the proof of (2).\qed

The main property of admissible sequences we will use is summarised in the following proposition, which is essentially contained in Proposition 3.1 in \cite{AFK}.

\begin{prop}[Proposition 3.1 in \cite{AFK}]\label{prop_AFK}
Given any $\eta\in(0,1)$, $h_*>0$ and $M>1$, there exists $\cA_1  = \cA_1(M)> 0$ such that for any $\cA > \cA_1$, there exists $C(h_*,\eta, \cA)>0$, such that for any irrational $\alpha$, any $(\cA, \cA, \cA^{22}, \cA^{21})-$admissible subsequence $(Q_k)$ of denominators of $\alpha$ with $Q_0=1$, and any $h>h_*$, any function $\phi\in C_h^{\omega}(\T,\R)$, it holds for any $k\in\Z_+$ and $h_k:=h(1-\eta k^{-2})$,
\begin{itemize}
\item $\|\phi^{[Q_k]}-Q_k \hat\phi(0)\|_{h_k}\leq C\|\phi-\hat\phi(0)\|_h(Q_k^{-M}+\overline{Q}_k^{-1+\frac{1}{M}})$,
\item for any $0 \leq l\leq Q_{k+1}$, $\|\phi^{[l]}-l \hat\phi(0)\|_{h_k}\leq C\|\phi-\hat\phi(0)\|_h(\overline{Q}_k Q_k^{-M}+\overline{Q}_k^{\frac{1}{M}})$.
\end{itemize}
 \end{prop}

\section{A KAM scheme for $SL(2,\R)$ cocycles}


In this section, we recall the KAM scheme for $SL(2,\R)$ cocycles developed in \cite{AFK}.

Given $\alpha\in\R\setminus\Q$ with $(q_n)_{n\geq 0}$ the sequence of denominators, and an open interval $J\subset\R$.
For $A:J\times\T\to SL(2,\R)$, analytic on $J\times \T_h$ for some $h>0$,
we consider the cocycle $(\alpha, A)$:
$$
\left(\begin{array}{c}
u_{n+1} \\[1mm]
u_{n}
\end{array}
\right)=A(E,\theta+n\alpha)\left(\begin{array}{c}
u_{n} \\[1mm]
u_{n-1}
\end{array}
\right).
$$
For $n\geq 1$, we denote the iterates of $(\alpha,A)$ by
\begin{eqnarray*}
A^{(n)}(E,\theta)&:=& A(E,\theta+(n-1)\alpha)\cdots A(E,\theta) , \\[1mm]
A^{(-n)}(E,\theta)&:=& A(E,\theta-n\alpha)^{-1}\cdots A(E,\theta-\alpha)^{-1}  .
\end{eqnarray*}
Denote $\rho(E) := \rho_{(\alpha, A(E,\cdot))}$ the fibered rotation number.

%


Before stating the main result in this section, we introduce some necessary objets.
Given $\tau > 0$, $0 < \nu < \frac12$ and $\varepsilon > 0$, we denote
$$\mathcal{Q}_{\alpha}(\varepsilon, \tau, \nu) := \left\{\rho \in \T : \|2q_n\rho\| > \varepsilon \max(q_n^{-\tau}, q_{n+1}^{-\nu}), \ \forall n \in \N \right\}.$$
We call $\mathcal{Q}_{\alpha}(\varepsilon, \tau, \nu)$ the set of {\bf non-resonant fibered rotation numbers} with parameters $\varepsilon$, $\tau$, $\nu$.
For any subsequence $(Q_k)$ of denominators, define
$${\cal Q}_l:=\{\rho\in\T: \|2Q_{l}\rho\| > \varepsilon \max( Q_{l}^{-\tau}, \  \overline{Q}_{l}^{-\nu})\}, \quad \overline{\mathcal{Q}}_l := \cap_{k=1}^{l} \mathcal{Q}_{k}, \quad \overline{\mathcal{Q}} := \cap_{k=1}^{\infty} \mathcal{Q}_{k},$$
and $\Omega_l := \rho^{-1}(\mathcal{Q}_l)$, $\overline{\Omega}_l := \rho^{-1}(\overline{\mathcal{Q}}_l)$, $\overline{\Omega} := \rho^{-1}(\overline{\mathcal{Q}}) = \cap_{k=1}^{\infty} \Omega_k$.

\begin{prop} \label{prop reducibility}  Given $h > 0$, an open interval $J \subset \R$,  let $A : J \times \T \to SL(2,\R)$, analytic on  $J \times \T_h$, and $A_0\in SO(2,\R)$.
Given $\tau > 0$, $0 < \nu < \frac12$, $\varepsilon > 0$,  integer $r \geq 1$, there exists ${\cA}_0={\cA}_0(\tau, \nu) > 0$, such that for any ${\cA} > {\cA}_0$,
%
%
one can find $\epsilon_0 = \epsilon_0(\tau, \nu, \varepsilon, h, r, {\cA}) > 0$ such that
if
$|A - A_0|_{J, \, h} < \epsilon_0 $, then the following is true for any $({\cal A},{\cal A},{\cal A}^{22},{\cal A}^{21})-$admissible sequence $(Q_k)$.
\begin{enumerate}
\item There exist $\left\{\begin{array}{l}
    W :(J \cap \overline{\Omega}) \times \T \to SL(2,\R)\\[1mm]
    \phi : (J \cap \overline{\Omega}) \times \T \to \T
  \end{array}\right.$, such that for each $E \in J \cap \overline{\Omega}$, $W(E, \cdot)$ and $\phi(E, \cdot)$ are analytic, satisfying
\begin{equation} \label{reduc}
A(E,\theta) = W(E, \theta + \alpha) R_{\phi(E,\theta)} W(E,\theta)^{-1}.
\end{equation}
\item For $l \geq 1$, there exist $\left\{\begin{array}{l}
   W_l : (J \cap \overline{\Omega}_l) \times \T \to SL(2,\R) \\[1mm]
   \phi_l : (J \cap \overline{\Omega}_l) \times \T \to \T\\[1mm]
   \xi_l : (J \cap \overline{\Omega}_l) \times \T \to \M(2,\R)
  \end{array}\right.$, analytic on $(J \cap \overline{\Omega}_l) \times \T$, such that
  $$A(E,\theta) = W_l(E,\theta + \alpha)R_{\phi_l(E,\theta)}\left(id + \xi_l(E,\theta)\right) W_l(E,\theta)^{-1}$$
 and there exist constants $c_1$, $D_1>0$, such that for every $l\geq 1$,
\begin{eqnarray}
  |W |_{J \cap \overline{\Omega}},\ |\phi |_{J \cap \overline{\Omega}},\    |W_l |_{C^{ r }(J \cap \overline{\Omega}_{l})},\  |\phi_l|_{C^{ r}(J \cap \overline{\Omega}_{l})} &<& c_1, \label{prop_error_1} \\
  |W_l - W_{l+1}|_{C^{ r }(J \cap \overline{\Omega}_{l+1})},\  |\phi_l - \phi_{l+1}|_{C^{ r }(J \cap \overline{\Omega}_{l+1})} &<& c_1  e^{-\overline{Q}_l^{D_1}},  \label{prop_error_2}\\
  |\xi_l|_{C^{ r}(J \cap\overline{\Omega}_l)} &<&  c_1 e^{-\overline{Q}_l^{D_1}}.  \label{prop_error_3}
  \end{eqnarray}
\item  If there exists $m_1 \in\Z_+$ and $\Lambda \in C^{\omega}(\T,SL(2,\R))$ such that
\begin{equation}\label{1-direction_condition}
\inf_{v \in \mathbb{P}(\R^2)\atop{E \in J, \ \theta \in \T}} \partial_{E}(\Lambda(\theta + m_1\alpha)A^{(m_1)}(E,\theta)\Lambda(\theta)^{-1} v)  > 0,
\end{equation}
then there exists  $C_1 > 0$ and $n_0$, $l_0\in\Z_+$ such that for any $n  > n_0$, $l > l_0$,
 $$\partial_{E}\phi_l^{[n]}(E, \theta) > C_1 n, \quad \forall \,  E \in J \cap \overline{\Omega}_l, \  \theta \in \T.$$
\end{enumerate}
\end{prop}

\noindent{\bf Proof of Proposition \ref{prop reducibility}:}
To prove Proposition \ref{prop reducibility}, the main step is to construct $W_l$ and $\phi_l$ by induction. By letting $l$ tend to infinity, we shall obtain $W$ and $\phi$ as their limits.

Given any constants $\kappa \in (0,1)$,  $D > 1$, we define the sequences
$$\kappa_l:=\frac{\kappa}{l^2}, \quad D_l: = D- \kappa_l, \quad h_l :=  h \prod_{i=1}^{l} (1 - \kappa_i)^2.$$
Let $h_{\infty} := \lim_{l \to \infty} h_l$. It is clear that $h_{\infty} > 0$.
 Let
$U_k := e^{-\overline{Q}_{k}Q_{k}^{-b} - \overline{Q}_{k}^{a}}$,  where $a = \frac{2}{M}$, $b = a^{-1}$, and
$$M  = \max\left(4(\tau+1),\, \frac{4}{1-2\nu} \right), \quad \cA_0 = {\cA}_1(M),$$
with ${\cA}_1(M)$ defined as in Proposition \ref{prop_AFK}.
The inductive step is given by the following lemma.

\begin{lemma}[Inductive lemma] \label{lemma induction}
Given any $\kappa\in(0,1)$, $D>1$, any $r\in\Z_+$, there exists $T_0(h, D, \kappa, \varepsilon, \nu, \tau, r)>0$, $C_0(h, D, \kappa, \varepsilon, \nu, \tau, r )>0$, such that the following is true.

For any $l$ such that $Q_{l} \geq T_0$, if there exist
$\left\{\begin{array}{l}
    W_l:(J \cap \overline{\Omega}_l) \times \T\rightarrow SL(2,\R)\\[1mm]
    \phi_l:(J \cap \overline{\Omega}_l) \times \T\rightarrow \T\\[1mm]
    \xi_l:(J \cap \overline{\Omega}_l) \times \T\rightarrow \M(2,\R)
 \end{array}\right.$, analytic on $(J \cap \overline{\Omega}_l) \times \T_{h_l}$, satisfying $\left|\phi_l - \hat{\phi}_l(0)\right|_{J \cap \overline{\Omega}_l, \, {h_l}} < D_{l}$, $|\xi_l|_{J \cap \overline{\Omega}_l, \, {h_l}} < U_l$, and
 $$A(E,\theta) = W_{l}(E,\theta + \alpha) R_{\phi_l(E,\theta)}\left(id + \xi_l(E,\theta)\right) W_l(E,\theta)^{-1},$$
then there exist
   $\left\{\begin{array}{l}
    B_l:(J \cap \overline{\Omega}_{l+1}) \times \T\rightarrow SL(2,\R)\\[1mm]
    \phi_{l+1}:(J \cap \overline{\Omega}_{l+1}) \times \T\rightarrow \R\\[1mm]
    \xi_{l+1}:(J \cap \overline{\Omega}_{l+1}) \times \T\rightarrow \M(2,\R)
 \end{array}\right.$, analytic on $(J \cap \overline{\Omega}_{l+1}) \times \T_{h_{l+1}}$, such that $W_{l+1} = W_{l}B_l^{-1}$ satisfies
 $$A(E,\theta) = W_{l+1}(E,\theta + \alpha) R_{\phi_{l+1}(E,\theta)}\left(id + \xi_{l+1}(E,\theta)\right) W_{l+1}(E,\theta)^{-1},$$
 with  $|\phi_{l+1} - \hat{\phi}_{l+1}(0)|_{J \cap \overline{\Omega}_{l+1}, \, {h_{l+1}}} < D_{l+1}$ and
 $|\xi_{l+1}|_{J \cap \overline{\Omega}_{l+1}, \, {h_{l+1}}} < U_{l+1}$. Moreover if
 $$(1 - 2^{-l})H_0 > 1+ \max_{ 1 \leq j \leq r }\left(| \partial^j_{E}\phi_l|^{\frac{ r }{j}}_{J \cap \overline{\Omega}_{l}, \, h_l}, \, |\partial^j_{E}\xi_{l}|^{\frac{ r }{j}}_{J \cap \overline{\Omega}_{l},\,  h_l} U_{l}^{-\frac{r}{j}}\right),$$
 for some constant $H_0$,
 then for  $0 \leq j \leq r$ ,
 $$|\partial_{E}^{j}(B_l-id)|_{J \cap \overline{\Omega}_{l+1}, \, h_{l+1}}, \  |\partial_{E}^{j}(\phi_{l+1}-\phi_{l})|_{J \cap \overline{\Omega}_{l+1}, \, h_{l+1}} \leq C_0  U_l^{\frac{1}{3}}  H_0^{\frac{j}{ r }},$$
 and $(1-2^{-l-1})H_0 > 1+ \max_{ 1 \leq j \leq r }\left(| \partial^j_{E}\phi_{l+1}|^{\frac{ r }{j}}_{J \cap \overline{\Omega}_{l+1}, \, h_{l+1}}, \,
 |\partial^j_{E}\xi_{l+1}|^{\frac{ r }{j}}_{J \cap \overline{\Omega}_{l+1},\,  h_{l+1}} U_{l+1}^{-\frac{r}{j}}\right)$.
\end{lemma}

\noindent {\bf Proof of Lemma \ref{lemma induction}.}
Let $r$ given as Lemma \ref{lemma induction}.
Lemma \ref{lemma induction} is a consequence of the following three lemmas.

\begin{lemma} \label{lemma iteration}[Lemma 4.5 of \cite{AFK}] Let $\phi_l, \xi_l$ be given as in Lemma \ref{lemma induction}
 and denote $A_l = R_{\phi_l}(id + \xi_l) $.
There exists $T = T(h, D, \kappa, \varepsilon, \nu, \tau, r )$ such that
if $Q_l \geq T$, then, on $\overline{\Omega}_{l+1}$,
 $A_{l}^{(Q_{l+1})}$ can be expressed as $R_{\phi_l^{[Q_{l+1}]}}(id + \xi_{(Q_{l+1})})$ with
 \begin{eqnarray*}
&&\left|\phi_l^{[Q_{l+1}]} - \widehat{ \phi_l^{[Q_{l+1}]} }(0) \right|_{J \cap \overline{\Omega}_{l+1}, \, h_l(1-\kappa_{l+1})}  \leq D,  \\
 &&  |(R_{2\phi_l^{[Q_{l+1}]}} - id)^{-1}|_{J \cap \overline{\Omega}_{l+1}, \, h_l(1-\kappa_{l+1})} < \varrho_l^{-1} := \frac{\varepsilon}{4(Q_{l+1}^{-\tau} + \overline{Q}_{l+1}^{-\nu})} < U_l^{-\frac{1}{40r}} , \\
&&|\xi_{(Q_{l+1})}|_{J \cap \overline{\Omega}_{l+1}, \, h_l(1-\kappa_{l+1})} \leq U_{l}^{\frac{1}{2}}  .
 \end{eqnarray*}
Moreover, if for $\overline{H} > 0$ we have $\overline{H} > 1+ \max_{ 1 \leq j \leq r }\left(| \partial^j_{E}\phi_l|^{\frac{ r }{j}}_{J \cap \overline{\Omega}_{l}, \, h_l},\, |\partial^j_{E}\xi_{l}|^{\frac{ r }{j}}_{J \cap \overline{\Omega}_{l},\,  h_l} U_{l}^{-\frac{r}{j}}\right)$, then

$$\overline{H} Q_{l+1}^r \gtrsim 1+  \max_{ 1 \leq j \leq r }\left( \left|\partial_{E}^j\phi_l^{[Q_{l+1}]}\right|^{\frac{ r }{j}}_{J \cap \overline{\Omega}_{l+1}, \, h_l(1-\kappa_{l+1})},\, |\partial_E^j \xi_{(Q_{l+1})}|^{\frac{ r }{j}}_{J \cap \overline{\Omega}_{l+1}, \, h_l(1-\kappa_{l+1})} U_{l}^{-\frac{r}{2j}}  \right).$$

\begin{proof}
The first three inequalities is contained in Lemma 4.5 in \cite{AFK}. The last inequality follows from direct computation. Here we have used the definition of $M$ and Proposition \ref{prop_AFK}.
\end{proof}
\end{lemma}

\smallskip


\begin{lemma} \label{lemma conjugation}
For every $D$, $h_{*} > 0$, there exist $\epsilon = \epsilon( D, h_{*})>0$, $C_0 = C_0(D, h_{*})>0$
such that the following is true.

Given any $\epsilon_0 \in (0, \epsilon)$, $h > h_{*}$, $0 < \delta < 1$, $\bar{\alpha} \in \R \setminus \Q$, given some open interval $J_0\subset\R$, and
$\left\{\begin{array}{l}
     \bar{A}: J_0 \times \T\rightarrow SL(2,\R) \\[1mm]
     \bar{\varphi}: J_0 \times \T\rightarrow\R
   \end{array}\right.$, analytic on $J_0 \times \T_h$, with
   $$|\bar{\varphi} - \widehat{\bar{\varphi}}(0)|_{J_0,\, h} \leq D, \quad \varrho^{-1}:= \max(4, |(R_{2\bar{\varphi}} -id)^{-1}|_{J_0,\, h} ) < \epsilon_0^{-\frac{1}{20r}}, \quad |R_{-\bar{\varphi}}\bar{A} -id|_{J_0,\, h}  < \epsilon_0.$$
Denote $\bar{\xi} = R_{-\bar{\varphi}}\bar{A} -id$. Then there exist $\left\{\begin{array}{l}
     B: J_0 \times \T\rightarrow SL(2,\R) \\[1mm]
     \varphi: J_0 \times \T\rightarrow\T
   \end{array}\right.$, analytic on $J_0 \times \T_{e^{-\delta/5}h}$, such that for $$\tilde{\bar{A}}(E,\theta):= B(E,\theta + \bar{\alpha}) \bar{A}(E,\theta) B(E,\theta)^{-1}$$
   and any constant
$H \geq 1+ \max_{1 \leq j \leq r }\left( |\partial_{E}^{j}\bar{\varphi}|^{\frac rj}_{J_0, \,h}, \,|\partial^j_{E}\bar{\xi}|^{\frac{r}{j}}_{J_0, \,h}\epsilon_0^{-\frac r j}\right) $,
we have
\begin{eqnarray}
|\partial_{E}^j(B-id)|_{J_0, \, e^{-\delta/5}h}&<&C_0 \epsilon_0 \rho^{-j-1} H^{\frac{j}{r}}, \quad j=0,\cdots, r, \label{lemma 6 term 101}\\
|R_{-\varphi}\tilde{\bar{A}} -id|_{J_0, \, e^{-\delta/5}h}&<& C_0  \epsilon_0    e^{-\frac{\delta h \varrho^{2r+1}}{C_0 \|\bar{\alpha}\|}}, \label{lemma 6 term 102} \\
|\mathcal{Q}(\partial_{E}^{j}\tilde{\bar{A}})|_{J_0, \, e^{-\delta/5}h}&<&  C_0  \epsilon_0  e^{-\frac{\delta h \varrho^{2r+1}}{C_0 \|\bar{\alpha}\|}}H^{\frac{j}{r}}, \quad j=1,\cdots, r. \label{lemma 6 term 103}
\end{eqnarray}
where $\mathcal{Q}(P) := \frac{P + JPJ}{2}$ for any $P \in \M(2,\R)$,  with $J := \left(\begin{array}{cc}
                                                                                    0 & -1 \\
                                                                                    1 & 0
                                                                                  \end{array}\right)$.
\end{lemma}

We leave the proof of Lemma \ref{lemma conjugation} to the appendix.

By Lemma \ref{lemma iteration}, we can
apply Lemma \ref{lemma conjugation} to $\bar A=A_l^{(Q_{l+1})}$, $\bar\alpha=Q_{l+1}\alpha$ and $\bar\varphi=\phi_l^{[Q_{l+1}]}$, with $J_0$ a connected component of $J \bigcap \overline{\Omega}_{l+1}$, $\epsilon_0=U_l^{\frac12}$, $h = h_l(1 - \kappa_{l+1})$, $h_{*} = h_{\infty}$, $H = c_{0} H_0Q_{l+1}^{r}$ for some large absolute constant $c_{0} >0$, and $\delta = - \log (1 - \kappa_{l+1})$, we obtain $B_ l = B$. The upper bound of $|\partial_{E}^{j}(B_l-id)|_{J \cap \overline{\Omega}_{l+1}, \, h_{l+1}}$ follows from \eqref{lemma 6 term 101}, $\rho^{20r} \geq \epsilon_0$ and $T_0 \leq Q_{l+1} \leq \overline{Q}_{l}^{\cA^{21}}$.

\begin{lemma}\label{lemma_derivative}
Let $h>h_*$, $D$, $\delta$, $\epsilon_0$, $\varrho$, $H$, $\bar{\alpha}$, $\bar{A}$, $B$, $\varphi$, $\tilde{\bar{A}}$ be given as in Lemma \ref{lemma conjugation}.
For any $\alpha \in \R \setminus \Q$ , there exists $D_0(h_*,D) > 0$ such that the following is true.

For any
$G:J_0 \times \T\rightarrow SL(2,\R)$, which is analytic on $J_0 \times \T_h$ and
satisfying $|G|_{J_0, \,h} \leq D$,  $G = R_{\zeta}(id + \xi)$ for some
$\left\{\begin{array}{l}
\zeta: J_0 \times \T\rightarrow \T \\[1mm]
\xi: J_0 \times \T\rightarrow \M(2,\R)
\end{array}\right.$, analytic on $J_0 \times \T_h$, with $|\xi|_{J_0, h} < \epsilon_0$, and $$H>1 + \max_{1 \leq j \leq r} \left( |\partial_{E}^j\zeta|^{\frac rj}_{J_0, \, h}, |\partial_E^{j}\xi|^{\frac r j}_{J_0, \, h}\epsilon_0^{-\frac{r}{j}}\right).$$
Moreover we have
$G(\theta + \bar{\alpha}) \bar{A}(\theta) = \bar{A}(\theta + \alpha) G(\theta)$. Then there exist
$\left\{\begin{array}{l}
\tilde\zeta: J_0 \times \T\rightarrow \T \\[1mm]
\tilde\xi: J_0 \times \T\rightarrow \M(2,\R)
\end{array}\right.$, analytic on $J_0 \times \T_{e^{-\delta}h}$, such that $\tilde{G}(\theta) := B(\theta + \alpha) G(\theta) B(\theta)^{-1}$ can be expressed as $\tilde{G} = R_{\tilde{\zeta}}(id + \tilde{\xi})$. Moreover, for $j=0,\cdots, r$, we have
\begin{equation}\label{estim_zeta_xi}
|\partial_{E}^{j}(\tilde{\zeta} - \zeta)|_{J_0, \, e^{-\delta}h} < D_0 H^{\frac{j}r}\epsilon_0, \quad |\partial_{E}^j \tilde{\xi}|_{J_0, \, e^{-\delta}h}< D_0H^{\frac{j}{r}}  e^{-\frac{\delta h \varrho^{2r+1}}{D_0  \|\bar{\alpha}\|}}.
\end{equation}
\end{lemma}

We leave the proof of Lemma \ref{lemma_derivative} to the Appendix \ref{appendix}.

\smallskip

Recall that  $\epsilon_{0} = U_l^{\frac{1}{2}}$. Apply Lemma \ref{lemma_derivative} to $G=A_l$, $\zeta=\phi_l$, $\xi=\xi_l$, with $J_0$ a connected component of $J \bigcap \overline{\Omega}_{l+1}$, we obtain $\phi_{l+1} = \tilde{\zeta}$, $\xi_{l+1} = \tilde\xi$ for Lemma \ref{lemma induction} and establish the desired estimates for $|\partial^{j}_{E}(\phi_{l+1} - \phi_l)|$, $j=0,\cdots, r$, by the first inequality in (\ref{estim_zeta_xi}) and $U_{l}^{1/6} \ll Q_{l+1}^{-r} $. Moreover, the second inequality in (\ref{estim_zeta_xi}) yields that
 $$|\xi_{l+1}|_{J \cap \overline{\Omega}_{l+1}, \, {h_{l+1}}}\leq D_0 e^{-\frac{ h(1-\kappa_{l+1})\kappa_{l+1}  \varrho_l^{2r+1}}{D_0 \|Q_{l+1}\alpha\|}} < U_{l+1},$$
in view of the proof of Proposition 4.7 in \cite{AFK}.

Finally, we verify the last statement in Lemma \ref{lemma induction}. Notice that by the second inequality in \eqref{estim_zeta_xi}, we have,
again by the proof of Proposition 4.7 in \cite{AFK}, for all $1 \leq j \leq r$, $Q_l \geq T_0$ for sufficiently large $T_0$,
\begin{eqnarray*}
|\partial_{E}^{j}\xi_{l+1}|_{J_0\cap \overline{\Omega}_{l+1},\, h_{l+1}} < D_0 H_0^{\frac{j}{r}}( c_0Q_{l+1}^{r} )^{\frac{j}{r}} e^{-\frac{ h(1-\kappa_{l+1})\kappa_{l+1}  \varrho_l^{2r+1}}{D_0 \|Q_{l+1}\alpha\|}} < D_0 H_0^{\frac{j}{r}} U_{l+1}.
\end{eqnarray*}
\qed


\

Now we are ready to prove Proposition \ref{prop reducibility}.
Recall that we have given $0<\kappa<1$, $D>1$ at the beginning of proof and $\tau>0$, $0<\nu<\frac12$, $\varepsilon>0$, $h>0$.
We choose $T_0=T_0(h, D, \kappa, \varepsilon, \nu, \tau, r)>0$ as in Lemma \ref{lemma induction}.
Let $\epsilon_0 < U_{l_0}$, where $l_0$ is the smallest number such that $Q_{l_0} > T_0$ with $(Q_k)$ the $({\cal A},{\cal A},{\cal A}^{22},{\cal A}^{21})-$admissible subsequence.
We can see that $\epsilon_0$ can be chosen only depending on $h$, $\kappa$, $\varepsilon$, $\tau$, $\nu$, $\cA, r$, while independent of the specific choice of the subsequence $( Q_l )$.

Since $ |A - A_0|_{J} < \epsilon_0 $ for some $A_0\in SO(2,\R)$,
let $W_{l}=id$, $R_{\phi_l}=A_0$ and $\xi_l=R_{-\phi_l}A-id$. We see that $\phi_l$ is constant in $\theta$, and $|\xi_l|_{J} < \epsilon_0$.
Then we can apply Lemma \ref{lemma induction} iteratively, starting on $J\cap\overline\Omega_{l_0}$, and obtain the sequences $\{B_l\}$, $\{\phi_{l}\}$ and $\{\xi_l\}$.
In view of Lemma \ref{lemma induction}, we have, for all $l \geq l_0$,
$$|\partial_{E}^{j}(B_l-id)|_{J \cap \overline{\Omega}_{l+1},\, h_{l+1}} < C_0 U_l^{\frac{1}{3}}H^{\frac{j}{r}}, \quad j=0,\cdots r.$$
Thus we can obtain $W$, $\phi$ as the limits of $W_l$, $\phi_l$ respectively as $l\rightarrow\infty$.
By Lemma \ref{lemma induction}, we obtain \eqref{prop_error_2}, \eqref{prop_error_3}, and hence \eqref{prop_error_1}.
Of course we can construct $W_l$ arbitrarily for $l < l_0$ at the cost of enlarging $c_1$ in $(\ref{prop_error_1})-(\ref{prop_error_3})$. Then (1) and (2) are proven.



We now prove (3). Define $C_2>0$ as
$$C_2:=\inf_{v \in \mathbb{P}(\R^2)\atop{E \in J, \ \theta \in \T}} \partial_{E}(\Lambda(\theta + m_1\alpha)A^{(m_1)}(E,\theta)\Lambda(\theta)^{-1} v),$$
Recall that, for each $l \geq 1$,
$A(E,\theta)=W_l(E,\theta+\alpha)A_l(E,\theta) W_l(E,\theta)^{-1}$ with
\begin{equation} \label{term Al}
A_{l}(E,\theta) = R_{\phi_l(E,\theta)}\left(id + \xi_l(E,\theta)\right).
\end{equation}
 So for any $m$, $l \geq 1$, $E \in J \cap \overline{\Omega}_l$, $\theta \in \T$, $v \in \mathbb{P}(\R^2)$, we have, with $\theta_{j}:=\theta+jm_1\alpha$ for $j\in\Z$,
\begin{eqnarray}
& &\partial_{E}(\Lambda(\theta_m)A^{(mm_1)}(E,\theta)\Lambda(\theta)^{-1}v) \nonumber\\
\quad  &= & \partial_{E}(\Lambda W_{l})(E, \theta_m)(A^{(mm_1)}_l(E,\theta) (\Lambda W_{l})(E,\theta)^{-1}v)\label{term1} \\
 &&+ \,  D(\Lambda W_{l})(E, \theta_m)\left(A^{(mm_1)}_l(E,\theta)(\Lambda W_l)(E,\theta)^{-1}v, \partial_{E}A^{(mm_1)}_l(E,\theta)(\Lambda W_l(E,\theta)^{-1}v)\right) \label{term2}\\
 && + \,  D(\Lambda W_{l}(E,\theta_m)A^{(mm_1)}_l(E,\theta))\left( \Lambda W_l(E,\theta)^{-1}v,  \partial_{E}(\Lambda W_l(E,\theta)^{-1})v\right).\label{term3}
\end{eqnarray}

Notice that $|D(\Lambda W_l)|_{J \cap \overline{\Omega}_l} \leq |\Lambda|^2 |W_l|^{2}_{J \cap \overline{\Omega}_l} \leq c_1^2|\Lambda|^2$.
By (2), we see that, for any $l$,
\begin{equation}
|(\ref{term1})|\leq c_1 |\Lambda|^2,\quad |(\ref{term3})|\leq c_1^5|\Lambda|^4|A^{(m m_1)}_l|_{J \cap \overline{\Omega}_l}^2. \label{term 101}
\end{equation}
By (2), for $m \geq 1$,  there exists $l(m) \geq 1$ such that $|A^{(mm_1)}_l|_{J \cap \overline{\Omega}_l} < 2$ for every $l \geq l(m)$.
On the other hand, we have
\begin{equation}
|\eqref{term2}|
\leq c_1^2 |\Lambda|^2 |\partial_EA_l^{(mm_1)}(E,\theta)((\Lambda W_l)(E, \theta)^{-1}v)|.
\label{term 102}
\end{equation}
By \eqref{reduc}, for any $E \in J \cap \overline{\Omega}$, and $k \in \Z$, we have
$$
|D(\Lambda(\theta_k)A^{(km_1)}(E,\theta)\Lambda(\theta)^{-1})| \geq |\Lambda|^{-4} |W|^{-4} \geq  |\Lambda|^{-4} c_1^{-4}.
$$
Thus, for any $E \in J \cap \overline{\Omega}$, $\theta \in \T$, $v \in \mathbb{P}(\R^2)$,
\begin{eqnarray}
|\partial_{E}(\Lambda(\theta_{m})A^{(mm_1)}(E,\theta)\Lambda(\theta)^{-1}v)|
&=& \sum_{k=1}^{m} \left|D(\Lambda(\theta_{m})A^{((m-k)m_1)}(E,\theta_k)\Lambda(\theta_{k})^{-1})(v_{k,1}, \, v_{k,2})\right|  \nonumber \\
&\geq& |\Lambda|^{-4} c_1^{-4} m C_2, \label{term 103}
\end{eqnarray}
with $v_{k,1}:=\Lambda(\theta_{k})A^{(km_1)}(E,\theta) \Lambda(\theta)^{-1}v$ and
$$ v_{k,2}:=\partial_{E}(\Lambda(\theta_{ k}) A^{(m_1)}(E, \theta_{k-1}) \Lambda(\theta_{k-1})^{-1})(v_{k-1,1}),$$
since $v_{k,2} \geq C_2$ for all $k$.

Using \eqref{term Al}, \eqref{term 101}, \eqref{term 102}, \eqref{term 103} and \eqref{prop_error_3}, by taking $m$ sufficiently large, we can ensure that for every $l \geq l(m)$,
$$|\partial_EA_l^{(mm_1)}(E,\theta)((\Lambda W_l)(E, \theta)^{-1}v)| > 1, \quad \forall \, E \in J \cap \overline{\Omega}, \ \theta \in \T, \ v \in \mathbb{P}(\R^2).$$
Then we conclude that $\partial_{E}A^{(mm_1)}_l(E, \theta)v > c_1^{-2}|\Lambda|^{-2}$ for all $E \in J \cap \overline{\Omega}$, $\theta \in \T$ and $v \in \mathbb{P}(\R^2)$. By (2), we see that $|\xi_l|_{C^1(J \cap \overline{\Omega}_l)} \to 0$ as $l\rightarrow\infty$. Thus by \eqref{term Al}, $\partial_{E}\phi_l^{(mm_1)} (E, \theta) > \frac{1}{2}c_1^{-2}|\Lambda|^{-2}$ for all $E \in J \cap \overline{\Omega}$, $\theta \in \T$ and sufficiently large $l \geq 1$. (3) follows from the uniform $C^2$ bound of $\phi_l$ on $J \bigcap \overline{\Omega}_l$ in \eqref{prop_error_1}.
\qed

\section{Growth of the diffusion norm}\label{growth}

Recall the time-dependent Schr\"odinger equation:
\begin{equation}\label{Eq_Schrodinger}
{\rm i}\frac{d}{dt}{u}_n=(L_\theta u)_n=-(u_{n+1}+u_{n-1})+V(\theta+n\alpha)u_n, \quad n\in\Z.
\end{equation}
In this section, we will finish the proof of Theorem \ref{main theorem 2}.

\subsection{Ballistic upper bound}
For Eq. (\ref{Eq_Schrodinger}), we already have the general ballistic upper bound(Lieb-Robinson bound\cite{LR}).
The proof can be found in \cite{AizenmanW} and \cite{DamanikLY}.
For the convenience of readers, we sketch the proof in \cite{DamanikLY}, and give the upper bound of the $p^{\rm th}-$moment of $u(t)$ for any $p>0$ under some suitable condition of $u(0)$.

\begin{thm}\label{ballistic_upper_general}
Given $p>0$. For $u(0)\in{\cal W}^{p'}(\Z)$ with $p'=\left\{\begin{array}{cl}
p,&\frac{p}{2}=\lfloor\frac{p}{2}\rfloor\\[2mm]
2\lfloor\frac{p}{2}\rfloor+2,&\frac{p}{2}>\lfloor\frac{p}{2}\rfloor
\end{array} \right.$,
there is a constant $C_3>0$, depending on $p$ and $\langle u(0)\rangle_{p'}$, such that
\begin{equation}\label{ballistic_upper_bound}
\langle u(t)\rangle_{p}\leq C_3 t^{\frac{p}2}, \quad t \geq 1.
\end{equation}
\end{thm}
\clb
\proof
Given $m\in\Z_+$, let $X^{(m)}$ be the multiplication operator from ${\cal W}^{2m}(\Z)$ to $\ell^{2}(\Z)$:  $(X^{(m)} u)_n = n^m u_n$, and for $N\in \Z_+$, we define the bounded operator $X_N^{(m)}$ on $\ell^2(\Z)$:
$$(X_N^{(m)} u)_n=\left\{ \begin{array}{cl}
-N^m u_n,& n\leq -N-1 \\[1mm]
-|n|^m u_n,& -N \leq n\leq 0\\[1mm]
n^m u_n, & 0< n\leq N\\[1mm]
N^m u_n, & n\geq N+1
\end{array} \right. .$$
Let $A^{(m)}_N={\rm i}[L, X^{(m)}_N]={\rm i}(L X^{(m)}_N-X^{(m)}_N L)$. By a straightforward computation, we have
$$(A_N^{(m)} u)_n=\left\{ \begin{array}{cl}
{\rm i}((n-1)^m-n^m)u_{n-1}+{\rm i}((n+1)^m-n^m) u_{n+1}, & |n| \leq N-1 \\[1mm]
{\rm i}(N^m-(N-1)^m)u_{-N+1},& n= - N\\[1mm]
-{\rm i}((N-1)^m-N^m)u_{N-1}, & n= N\\[1mm]
0, & |n|\geq N+1
\end{array} \right. .$$
Note that $(n+1)^m-n^m=\sum_{j=1}^m C^j_{m} n^{m-j}$ for any $n\in\Z$ and the constants $C_m^{j}$ are independent of $n$.
Then we see that $A^{(m)}_N$ converges strongly to the self-adjoint operator from ${\cal W}^{2m-2}(\Z)$ to $\ell^2(\Z)$:
$$(A^{(m)} u)_n={\rm i}((n-1)^m-n^m) u_{n-1}+{\rm i}((n+1)^m-n^m) u_{n+1},$$
as $N\rightarrow\infty$, i.e.,
$$\lim_{N \to \infty} \| A^{(m)}_{N} u - A^{(m)}u \|_{\ell^2(\Z)} = 0, \quad \forall u \in \cH^{2m-2}(\Z).$$
Moreover there is a constant $c'=c'(m)$, such that
$$
\|A^{(m)} u\|_{{\ell}^2(\Z)}, \|A^{(m)}_N u\|_{{\ell}^2(\Z)}\leq c' \lan u \ran_{2m-2},  \quad \forall N \in \Z_{+}.
$$

We consider the Heisenberg time evolution $X^{(m)}_N (t) := e^{{\rm i}tL}X^{(m)}_N e^{-{\rm i}tL}$ on $\ell^2(\Z)$ for $N \in \Z_{+}$.
Since $L$ is a bounded self-adjoint operator,
we can see that $e^{\pm{\rm i}tL}$ are analytic functions of $t$. So $X_N^{(m)} (t)$ is analytic in $t$.
Differentiating with respect to $t$, we have
$$\frac{d}{dt}X^{(m)}_N(t)=e^{{\rm i}tL}({\rm i} L)X_N^{(m)} e^{-{\rm i}tL}+e^{{\rm i}tL}X_N^{(m)}(-{\rm i} L) e^{-{\rm i}tL}=e^{{\rm i}tL}A_N^{(m)} e^{-{\rm i}tL}:=A^{(m)}_N(t).$$
For $T>0$, integrating with respect to $t$ on $[0,T]$, we have
$$X_N^{(m)}(T)-X_N^{(m)}=\int_0^T A^{(m)}_N(t) dt.$$
For any $\psi\in {\cal W}^{2m}(\Z)$, assume that $e^{-{\rm i}tL}\psi\in {\cal W}^{2m-2}$ for any finite $t$. This is true for $m=1$ in view of the $\ell^2-$conservation law.
Then we can define $$A^{(m)}(t) : = e^{{\rm i}tL}A^{(m)} e^{-{\rm i}tL}.$$
Noting that $e^{{\rm i}tL}$ is unitary on $\ell^2(\Z)$, we have $\|A^{(m)}(t)\psi\|_{{\ell}^2(\Z)}\leq c' \lan e^{-{\rm i}tL}\psi\ran_{2m-2}$ for any $t>0$.
Combining with the strong convergence of $A_N^{(m)}(t)$ to $A^{(m)}(t)$, we get, by dominated convergence theorem, that for any $T > 0$, any $\psi \in \cH^{2m}(\Z)$,
$$\lim_{N\rightarrow\infty}(X_N^{(m)}(T)\psi-X_N^{(m)}\psi)=\lim_{N\rightarrow\infty}\int_0^T A^{(m)}_N(t)\psi \, dt=\int_0^T A^{(m)}(t)\psi \, dt,$$
where the limit is taken in space $\ell^2(\Z)$.
So we can define for  $\psi\in {\cal W}^{2m}(\Z)$ that $X^{(m)}(T)\psi :=\lim_{N\rightarrow\infty}X_N^{(m)}(T)\psi$
, with
\begin{eqnarray*}
\lan e^{-{\rm i}TL}\psi \ran_{2m} = \|X^{(m)}(T)\psi\|_{\ell^2(\Z)}&\leq& \|X^{(m)}\psi\|_{\ell^2(\Z)}+\int_0^T \|A^{(m)}(t)\psi\|_{\ell^2(\Z)} \, dt\\
&\leq&\lan \psi \ran_{2m}+c' \int_0^T \lan e^{-{\rm i}tL}\psi \ran_{2m-2}  \, dt.
\end{eqnarray*}

Since for $m=1$, $\lan e^{-{\rm i}tJ}\psi\ran_{2m-2}= \sqrt{2}\|\psi\|_{\ell^2(\Z)}$ for any finite $t$, we can get the linear upper bound for $\lan e^{-{\rm i}tJ}\psi \ran_{2}$. By the induction, for any $m\in\Z_+$, there is a constant $c_m>0$,
depending on $m$ and $\langle\psi\rangle_j$, $j=0,\cdots, m$, such that
$$\lan e^{-{\rm i}TL}\psi \ran_{2m}\leq c_m T^m, \quad \forall \, 1\leq T <\infty.$$
This concludes the proof of (\ref{ballistic_upper_bound}) when $p$ is an even integer.

For any $p\in (2m, 2m+2)$, with $u(0)\in {\cal W}^{2m+2}$, we have
$$\langle u(t) \rangle_{2m}\leq c_m t^{m}, \quad \langle u(t) \rangle_{2m+2}\leq c_{m+1} t^{m+1}.$$
By Cauchy inequality $\langle u(t) \rangle_{p} \leq \langle u(t) \rangle_{2m}^{\frac{2m+2-p}{2}} \langle u(t) \rangle_{2m+2}^{\frac{p-2m}{2}}$, we get (\ref{ballistic_upper_bound}).\qed
%
%

\

Theorem \ref{ballistic_upper_general} gives a stronger upper bound than that in Theorem \ref{main theorem 2}.
The rest part of this section is devoted to show that for any $\eta > 0$
$$\lim_{T\rightarrow\infty}\frac{\lan u(t)\ran^2_p}{t^{p-\eta}}=\infty$$
under the condition of Theorem \ref{main theorem 2}.
Given $\alpha\in\R\setminus\Q$, and $V\in C^\omega(\T,\R)$ such that $L_\theta$ has purely a.c. spectrum for a.e. $\theta\in\T$.
Let $\T_{ac}\subset \T$ denote this full-measure subset.

\smallskip

\subsection{Reducibility for the Schr\"odinger cocycle}

Given any integer $r\geq 3$.
we define the constants $\tau, \tau_0, \tilde{\tau}_0$, $\tau_1$, $\tau_2 > 100$ such that
\begin{equation}\label{tau_1tau_2}
\tilde{\tau}_0 \geq \tau_2 > \frac{10^{7}\tau_0}{\eta}, \quad \tau_0 \geq \tau_1, \quad \tau = \min(\tau_0, \tilde{\tau}_0).
\end{equation}
For any $0 < \nu < \frac12$, we choose $\cA$ such that
\begin{equation}\label{cal_A}
\cA > \max\left( \frac{10^{7}\tilde{\tau}_0}{\eta\tau_1}, \, \frac{10^{7}\tilde{\tau}_0}{\eta\nu}, \, \cA_0(\tau_0, \nu,r), \, \cA_0(\tilde{\tau}_0, \nu,r), \,\cA_1(\max(\tau_0, \tilde{\tau}_0) +1), \, \cA_1(\frac{1}{1-2\nu}) \right),
\end{equation}
with $\cA_0$ given by Proposition \ref{prop reducibility}, $\cA_1$ given by Proposition \ref{prop_AFK}.
Then, according to Lemma \ref{lemma_Qk} and \ref{lemma two sequences},
we can define the two $(\cA, \cA, \cA^{22},\cA^{21})-$admissible subsequences $(Q_{l})$ and $(R_{l})$.

Consider the quasi-periodic Schr\"odinger cocycle $(\alpha, A_{(E,V)})$:
$$
\left(\begin{array}{c}
u_{n+1} \\[1mm]
u_{n}
\end{array}
\right)=A_{(E,V)}(\theta+n\alpha)\left(\begin{array}{c}
u_{n} \\[1mm]
u_{n-1}
\end{array}
\right)  \  {\rm with}  \  A_{(E,V)}(\theta)=\left(\begin{array}{cc}
            -E+V(\theta) & -1 \\[1mm]
            1 & 0
          \end{array}
\right).
$$
Its fibred rotation number $\rho=\rho_{(\alpha, A_{(E,V)})}:\R\rightarrow [0,\pi]$ is monotonic.
Given $\varepsilon>0$, we define ${\Omega}_{l}$, $\overline{\Omega}_{l}$ and $\overline{\Omega}$ for the subsequence $(Q_{l})$ and parameters $(\varepsilon, \tau_0, \nu)$ as in the previous section. We can see that ${\Omega}_{l}$ is a union of at most $2Q_{l}$ intervals.
 Let $\Sigma_0:= \rho^{-1}(\mathcal{Q}_{\alpha}(\varepsilon, \tau, \nu))$.
 It is clear that $\Sigma_0\subset \overline{\Omega}$ since, by \eqref{tau_1tau_2}, $\tau \leq \tau_0$.

\smallskip

For any $E$ such that the Lyapunov exponent vanishes, it is natural to ask if the associated cocycle is analytically conjugated to constant matrix.
As a related concept, ``almost reducibility" was introduce by Eliasson\cite{E92}, and has been developed in \cite{A} and \cite{AFK} for one-frequency models.
Roughly speaking, a cocycle is said to be almost
reducible if, by a sequence of transformations, the transformed cocycles are closer and closer to constant.
\begin{definition}[almost reducible] Given $A_0\in SL(2,\R)$, the cocycle $(\alpha, A_{(E,V)})$ is called almost reducible to $A_0$, if there exist $h>0$, a sequence
$\{\Lambda_n\}_{n\geq0}\subset C^{\omega}(\T, SL(2,\R))$, analytic on $\T_h$,  such that
$$ \lim_{n \to \infty} |\Lambda_n(\cdot + \alpha)^{-1}A_{(E,V)}(\cdot)\Lambda_n(\cdot) - A_0|_h  = 0.$$
\end{definition}

By Theorem 1.4 and Corollary 1.6 of \cite{A}, we have
\begin{lemma}\label{lemma_ar}
For $\alpha\in\R\setminus \Q$ and $V\in C^{\omega}(\T, \R)$ such that $L_{\theta}$ has purely a.c. spectrum for a.e. $\theta$, then for a.e. $E\in\sigma(L)$, the cocycle $(\alpha, A_{(E,V)})$ is almost reducible to some $A_0\in SO(2,\R)$.
\end{lemma}

For more on reducibility or almost reducibility for $SL(2,\R)$ and $sl(2,\R)$ cocycles, and the development of KAM theory in this context,
we can refer to  \cite{DinaburgS, FK, HY, YZ}.

\

According to Lemma \ref{lemma_ar}, for any $E_0$ in a co-null set of $\Sigma_0$, there exists $h=h(E_0)$ such that
we can find $\Lambda \in C^{\omega}(\T, SL(2,\R))$, analytic on $\T_h$, satisfying
$$|\Lambda(\cdot + \alpha)^{-1} A_{(E_0,V)}(\cdot)\Lambda(\cdot) - A_0|_h< \frac{\epsilon_0}{2},$$
where $\epsilon_0 = \epsilon_0(\tau,\nu,\varepsilon, h, {\cA}) $ is given in Proposition \ref{prop reducibility}.
Then there exists an interval $J$ containing $E_0$ such that for every $E \in J$,
 $$|\Lambda(\cdot + \alpha)^{-1}A_{(E,V)}(\cdot)\Lambda(\cdot) - A_0|_h < \epsilon_0.$$
With the $(\cA, \cA, \cA^{22}, \cA^{21})-$admissible subsequence $(Q_l)$ given in the beginning of this section, by applying Proposition \ref{prop reducibility} to the cocycle $ (\alpha, \Lambda(\cdot+ \alpha)^{-1}A_{(\cdot,V)}(\cdot)\Lambda(\cdot))$ parametrised by $E\in J$,
we can find
$\left\{\begin{array}{l}
W:(J\cap\overline\Omega)\times \T \to  SL(2,\R)\\[1mm]
\phi:(J\cap\overline\Omega) \times \T \to  \T
\end{array}\right.$, such that
$$A_{(E,V)} = \Lambda(\cdot + \alpha)W(E, \cdot+ \alpha) R_{\phi(E,\cdot)} W(E,\cdot)^{-1} \Lambda^{-1}, \quad \forall E \in J \cap \overline{\Omega}.$$
Thus for every $E\in J \cap \overline{\Omega}$, the corresponding generalized eigenvector is in $\ell^\infty(\Z)$.

For any $\theta \in \T_{ac}$, we define $e_0 = e_0(\theta, u(0))$ as the Lebesgue measure of the subset
\begin{equation}\label{subset_e0}
\left\{E \in \sigma(L) \left| \begin{array}{l}
                                                   E \in \rho^{-1} (\mathcal{Q}_{\alpha}(\varepsilon, \tau, \nu) = \Sigma_0, \ {\rm and \ there \ exists \ a \ generalised \ eigenvector} \\[1mm]
                                                g(E)\in\ell^\infty(\Z) \  {\rm of} \ L_{\theta} \ {\rm such \ that } \ \sum_{n} u_n(0)g_n(E) \neq 0
                                                 \end{array}\right.\right\}.
\end{equation}
By choosing $\varepsilon$ small, combining with the absolute continuity of spectrum, we can ensure that $e_0>0$.

\smallskip

Note that $E_0$ is chosen arbitrarily in a co-null set of $\Sigma_0$, we can find finitely many open intervals $\{J_k\}_{1 \leq k \leq K}$ with $\{cl\{J_k\}\}$ mutually disjoint and $|\Sigma_0 \setminus \cup_{k} J_k | < \frac{e_0}{4}$, such that there is a constant $h > 0$ and $\left\{\begin{array}{l}
           \Lambda:(\cup_{k} J_k)\times \T \to  SL(2,\R)\\[1mm]
           A_0: \cup_{k} J_k \to  SO(2,\R)
          \end{array}\right.$ satisfying
\begin{itemize}
\item both of $\Lambda$ and $A_0$ are constant on each $J_k$;
\item $\Lambda$ is analytic on $\T_h$ for any $E\in\cup_{k} J_k$;
\item $|\Lambda(\cdot + \alpha)^{-1}A_{(E,V)}(\cdot)\Lambda - A_0|_{\cup_k J_k, \,  h}< \epsilon_0$.
\end{itemize}
Indeed, we can choose a compact set $\Sigma_1 \subset \Sigma_0$ such that $\Sigma_1$ contains only almost reducible energy and $|\Sigma_0 \setminus \Sigma_1| < \frac{e_0}{8}$. For each $E_0 \in \Sigma_1$, we choose the interval $J$ containing $E_0$ described as above. Then by compactness, we can choose finitely many such intervals to cover $\Sigma_1$. Then we choose a disjoint finer covering and choose $A_0(E,\cdot)$, $\Lambda(E,\cdot)$  according to which interval $E$ originally belongs to. Then we adjust $h$ accordingly.

Apply Proposition \ref{prop reducibility}(1) to $\Lambda(\cdot + \alpha)^{-1}A_{(\cdot,V)} \Lambda$ on each $J_k$,
we get $W$ and $\phi$.
Then we can define
$$Z: \left(\left(\cup_k J_k\right) \cap \overline{\Omega}\right)\times \T \to  SL(2,\R) \;  {\rm with} \; Z(E, \cdot) := \Lambda(\cdot) W(E, \cdot). $$
By Proposition \ref{prop reducibility}(2), we get the approximants $W_l$ and $\phi_l$.
So we also define
$$Z_l: \left(\left(\cup_k J_k\right) \cap \overline{\Omega}_l \right)\times \T \to  SL(2,\R) \;  {\rm with} \; Z_l(E, \cdot) := \Lambda(\cdot) W_l(E, \cdot).$$
By the monotonicity of the Schr\"odinger operators, it is direct to check that the condition (\ref{1-direction_condition}) in Proposition \ref{prop reducibility}(3) is satisfied with $\Lambda$ and $A(E,\cdot) = \Lambda(\cdot +\alpha)^{-1}A_{(E, V)}\Lambda$ given as above with $m_1=2$, thus we can find $n_0$, $l_0\in\N$, $C_1  > 0$ such that for $n>n_0$, $l>l_0$, $1 \leq k \leq K$
\begin{equation}\label{partial_phi_1}
\partial_E \phi_{l}^{[n]}(E,\theta) > C_1 n, \quad \forall \,  E \in J_k \cap \overline{\Omega}_l, \  \theta \in \T.
\end{equation}
Moreover there exists $c_1 > 0$, for any $l \geq 1$, any $n\in\Z$, any $1\leq k \leq K$,
\begin{equation}\label{partial_phi_2}
\left|\phi_{l}^{[n]}\right|_{C^r(J_k\cap \overline{\Omega}_{l})}\leq c_1 |n|.
\end{equation}

%

\begin{rema}
We have $(\cup_{k} J_k) \cap \overline{\Omega} \subset \sigma(L)$ by the (rotational)  reducibility.
\end{rema}

\begin{rema}\label{rmq_tilde}
Recall that we have constructed another $(\cA, \cA, \cA^{22}, \cA^{21})-$admissible sequence $(R_l)$ in Lemma \ref{lemma two sequences}.
We can define the subsets $\Gamma_l$, $\overline{\Gamma}_l$ and $\overline{\Gamma}$ for $(R_l)$ and parameters $(\varepsilon, \tilde{\tau}_0, \nu)$ in the same way as  $\Omega_l$, $\overline{\Omega}_l$ and $\overline{\Omega}$. We have $\Sigma_0 \subset \overline{\Gamma}$ since by \eqref{tau_1tau_2} $\tau \leq \tilde{\tau}_0$.
Then we can apply Proposition \ref{prop reducibility} to $(R_l)$ on $J_k\cap \overline{\Gamma}$(and $J_k\cap \overline{\Gamma}_l$), and get the conjugation matrix $\tilde W$, angle function $\tilde{\phi}$ and their approximants $\tilde{W}_l$, $\tilde{\phi}_l$. So we can define $\tilde Z(E, \cdot) := \Lambda(\cdot) \tilde W(E, \cdot)$, $\tilde Z_l(E, \cdot) := \Lambda(\cdot) \tilde W_l(E, \cdot)$.
\end{rema}

\smallskip
The following measure estimate is based on the estimates obtained in Proposition \ref{prop reducibility} and the monotonic property of Schr\"odinger cocycles.

\begin{lemma}\label{cor_measure}
There exists $C_4> 0$ such that for sufficiently large $l$, we have $$|(\cup_k J_k ) \cap (\overline{\Omega}_{l} \setminus \overline{\Omega})|  \leq C_4\max(Q_{l+1}^{-\tau_0}, \overline{Q}_{l+1}^{-\nu}).$$
\end{lemma}
\proof Recall that by applying Proposition \ref{prop reducibility} to $\Lambda(\cdot + \alpha)^{-1}A_{(E,V)}(\cdot)\Lambda(\cdot)$ on each $J_k$, with parameters $\varepsilon, \tau_0, \nu$, we obtain $W_m$, $\phi_m$ as above with
$$A_{m}(E,\cdot) =  W_m(E, \cdot + \alpha)^{-1} \Lambda(\cdot + \alpha)^{-1} A_{(E,V)} \Lambda(\cdot) W_m(E, \cdot) = R_{\phi_m(E,\cdot)}(id + \xi_{m}(E,\cdot) ).$$

By Claim 4.6 in \cite{AFK}, we have, for any $\theta\in\T$, and $1 \leq l \leq Q_{m+1}$,
$$A^{(l)}_{m}(E, \theta) = R_{\phi_m(E,\theta+(l-1)\alpha)}(id+\xi_m(E,\theta+(l-1)\alpha))\cdots R_{\phi_m(E,\theta)}(id+\xi_m(E,\theta))$$
can be expressed as $R_{\phi^{[l]}_m(E,\theta)}(id+\xi^{(l)}(E,\theta))$ such that for $1 \leq l \leq Q_{m+1}$,
 $$|\xi^{(l)}(E,\cdot)|_{J \bigcap \overline{\Omega}_l, h_l}\leq \exp\left\{{\sum_{k=0}^{l-1} |\xi_m(E,\cdot+k\alpha)|}\right\}-1 < 10 c_1Q_{m+1}e^{-\overline{Q}_{m}^{D_1}}$$
For $l=Q_{m+1}$, by Proposition \ref{prop_AFK} (apply to $\eta = 1$, $h = h_l > h_{\infty}$) and \eqref{cal_A}, \eqref{partial_phi_2}, when $m$ is sufficiently large, $C = C(h_{\infty}, 1/2, \cA)$ in Proposition \ref{prop_AFK}
$$\left|\phi_{m}^{[Q_{m+1}]} - Q_{m+1}\hat\phi_{m}(0)\right|_{ J_k  \cap \overline{\Omega}_m} < C \cdot c_1(Q_{m+1}^{-\tau_0 -1}, \overline{Q}_{m+1}^{-2\nu})  < \frac{\varepsilon}{100}\max(Q_{m+1}^{- \tau_0} , \overline{Q}_{m+1}^{-\nu}).$$

Note that the fibered rotation number of the cocycle $(Q_{k+1}\alpha, A^{(Q_{k+1})})$ is $Q_{k+1}\rho$.
Then for each $E \in J_k$, there exists $\theta \in \R$ such that
\begin{eqnarray*}
\widehat{A_{m}^{[Q_{m+1}]}}(E, \theta) - \theta = Q_{m+1}\rho(E),
\end{eqnarray*}
where $\widehat{A_{m}^{[Q_{m+1}]}}(E, \cdot) : \R \to \R$ denotes a lift of the projective map $A_{m}^{[Q_{m+1}]}(E, \cdot):\T \to \T$.
Hence
$$\left|\phi_{m}^{[Q_{m+1}]} - Q_{m+1}\rho\right|_{ J_k  \cap \overline{\Omega}_m}  \leq \frac{\varepsilon}{10}\max(Q_{m+1}^{- \tau_0} , \overline{Q}_{m+1}^{-\nu}) + 10 c_1Q_{m+1}e^{-\overline{Q}_{m}^{D_1}}.$$
Thus $J_k \cap (\overline{\Omega}_m \setminus \overline{\Omega}_{m+1}) \subset \left\{ E \in J_k \cap \overline{\Omega}_m | \|2 \phi_{m}^{[Q_{m+1}]}(E,\cdot)\| \leq 2\varepsilon \max( Q_{m+1}^{-\tau_0}, \overline{Q}_{m+1}^{-\nu}) \right\}$. Moreover by Proposition \ref{prop reducibility}(3), we have $c_1 Q_{m+1} \geq \partial_{E} \phi_{m}^{[Q_{m+1}]}(E, \cdot) > C_1 Q_{m+1}$ for $E \in J_k \cap \overline{\Omega}_m$ for any sufficiently large $m$.
Hence
$$\left|(\cup_k J_k) \cap (\overline{\Omega}_m \setminus \overline{\Omega}_{m+1})\right|\leq \frac{c_1}{2C_1}\max(Q_{m+1}^{-\tau_0}, \overline{Q}_{m+1}^{-\nu}).$$
 Thus $|(\cup_k J_k ) \cap (\overline{\Omega}_{l} \setminus \overline{\Omega})|  \leq C_4\max(Q_{l+1}^{-\tau_0}, \overline{Q}_{l+1}^{-\nu})$ for $C_4 = \frac{c_1}{C_1}$ when $l$ is sufficiently large.
\qed

\subsection{A measure supported on the spectrum with smoothing properties}

\begin{prop}\label{prop_psi}
There is a numerical constant $c_2>0$, and a sequence of $C^{\infty}$ functions $\{\psi_{(l)}\}_{l \geq 1}$ satisfying that $0 \leq \psi_{(l)} \leq 1$ and  $|\psi_{(l)}|_{C^{2}} \leq c_2 Q_{l}^{\tau_1}$, such that
$B\left(supp(\psi_{(1)}), Q_1^{-\tau_0} \right) \subset \overline{\Omega}_{1}$,
and for $l\geq 2$, with
$\psi_{l}:=\prod_{k=1}^l \psi_{(k)}$,
\begin{equation}\label{properties}
B\left(supp(\psi_{(l)}), Q_l^{-\tau_0} \right) \subset \overline{\Omega}_{l} \cap supp(\psi_{l-1}), \quad \left| \{ E \in \overline{\Omega}_l \cap supp(\psi_{l-1}) | \psi_{(l)}(E) \neq 1\} \right| \leq c_2 Q_{l}^{-\frac{\tau_1}{4}}.
\end{equation}
Moreover, with $\psi := \prod_{l=1}^{\infty}\psi_{(l)}$, we have $supp(\psi)\subset \overline{\Omega}$ and
\begin{equation}\label{error_psi}
|\psi - \psi_{l}|_{L^{1}((\cup_k J_k) \cap \overline{\Omega})} \leq c_2( Q_{l+1}^{-\frac{\tau_1}{4}} + \overline{Q}_{l+1}^{-\frac{\nu}{2}}).
\end{equation}
\end{prop}
\proof At first, we focus on these connected components of $\overline{\Omega}_{1}={\Omega}_{1}$. Since ${\Omega}_{1}$ is a union of at most $2Q_{1}$ intervals, for those intervals of length $<4Q_{1}^{-\frac{\tau_1}2}$, the measure of their union is less than $$8Q_{1}^{1-\frac{\tau_1}2} < \frac{1}{2}Q_{1}^{-\frac{\tau_1}{4}}.$$
On each interval of length $\geq 4Q_{1}^{-\frac{\tau_1}2}$, denoted by $J$, we can construct a $C^{\infty}$ function $\psi^{(1)}_J$ supported on $J$, satisfying $0 \leq \psi^{(1)}_J \leq 1$ and $|\psi^{(1)}_J|_{C^2}\leq c_2Q_{1}^{\tau_1}$, such that
$$B\left(supp(\psi^{(1)}_J), Q_{1}^{-\tau_0}\right) \subset J,\quad \left|\{ E \in J | \psi^{(1)}_{J}(E) \neq 1\}\right|\leq 4Q_{1}^{-\frac{\tau_1}2}.$$
Note that here we have used the relation $\tau_0 \geq \tau_1$.
Let $\psi_{(1)}:= \sum_{J} \psi^{(1)}_J$ where $J$ runs over all the connected components of length $\geq4Q_{1}^{-\frac{\tau_1}2}$.

Assume that we have already constructed $\psi_{(1)}$, $\cdots$, $\psi_{(l)}$ satisfying the above properties in (\ref{properties}). Moreover, we assume that for each $1 \leq k \leq l$, $supp(\psi_{(k)})$ is a union of at most $2kQ_{k}$ intervals. Then we can see that $supp(\psi_{l})$ has at most $2l^2Q_{l}$ connected components.
Since ${\Omega}_{l+1}$ is a union of at most $2Q_{l+1}$ intervals, we can see $\overline{\Omega}_{l+1} \cap supp(\psi_{l})$
 has at most $2(l+1)Q_{l+1}$ connected components by noting that $Q_{l+1} > l Q_{l}$.

We focus on these connected components of $\overline{\Omega}_{l+1} \cap supp(\psi_{l})$. For those intervals of length $<4Q_{l+1}^{-\frac{\tau_1}2}$, the measure of their union is less than $8(l+1)Q_{l+1}^{1-\frac{\tau_1}2} < \frac{1}{2}Q_{l+1}^{-\frac{\tau_1}{4}}$.
On each interval of length $\geq 4Q_{l+1}^{-\frac{\tau_1}2}$, saying $J$, we can construct a $C^{\infty}$ function $\psi^{(l+1)}_J$ supported on $J$, satisfying $0 \leq \psi^{(l+1)}_J \leq 1$ and $|\psi^{(l+1)}_J|_{C^2}\leq c_2Q_{l+1}^{\tau_1}$, such that
$$B\left(supp(\psi^{(l+1)}_J), Q_{l+1}^{-\tau_0}\right) \subset J,\quad \left|\{ E \in J | \psi^{(l+1)}_{J}(E) \neq 1\}\right|\leq 4Q_{l+1}^{-\frac{\tau_1}2}.$$
Let $\psi_{(l+1)}:= \sum_{J} \psi^{(l+1)}_{J}$ where $J$ runs over all the connected components of length $\geq4Q_{l+1}^{-\frac{\tau_1}2}$.
It is direct to check that $\psi_{l+1}$ satisfies the induction assumption.

Since for each $m$, we have $\left|\{ E \in \overline{\Omega}_m \cap supp(\psi_{m-1}) | \psi_{(m)}(E) \neq 1\}\right| \lesssim Q_{m}^{-\frac{\tau_1}{4}}$.
Combined with Lemma \ref{cor_measure}, (\ref{error_psi}) follows from the inclusion
$$\{ E: \psi(E) \neq \psi_{l}(E) \} \subset(\overline{\Omega}_{l} \setminus \overline{\Omega}) \cup \left(\cup_{m=l+1}^{\infty}  \{ E \in \overline{\Omega}_m \cap supp(\psi_{m-1}) | \mbox{ $\psi_{(m)}(E) \neq 1$}\} \right).$$
\qed

%
%

It is direct to see that we have
$|\psi_{l}|_{C^2(\overline{\Omega}_l)}\lesssim  l^2 Q_l^{\tau_1}$
and
$$supp(\psi)\subset \overline{\Omega}, \quad \left|\{ E \in  \overline{\Omega} | \psi(E) =0 \}\right| \lesssim \sum_{l=1}^{\infty} Q_{l}^{-\frac{\tau_1}{4}} \lesssim Q_{1}^{-\frac{\tau_1}{4}}.$$
Recall that $\Sigma_0 = \rho^{-1}(\mathcal{Q}_{\alpha}(\varepsilon, \tau, \nu)) \subset \overline{\Omega}$.
After possibly modifying $\psi_{(l)}$ for finitely many index $l$ (at the cost of enlarging the constant $c_2$ in Proposition \ref{prop_psi}), we can assume that
\begin{eqnarray} \label{supp positive}
&&\left|\left\{E \in (\cup_k J_k) \cap \Sigma_0 \left| \begin{array}{l}
                                                   \psi(E) > 0, \  \exists \ {\rm a \ \ell^{\infty} \ generalised \ eigenvector} \ g(E) \\[1mm]
                                                  {\rm of} \ L \ {\rm such \ that } \ \sum_{n} u_n(0)g_n(E) \neq 0
                                                 \end{array}\right.\right\} \right| > \frac{e_0}{2},
\end{eqnarray}
recalling that $e_0$ is the Lebesgue measure of the subset given in (\ref{subset_e0}).
\begin{rema}\label{rmq_tilde_psi}
The function $\psi$ will serve as a smoothing measure for the modified spectral transformation.
Similarly as in Remark \ref{rmq_tilde}, with $(Q_l)$ replaced by $(R_l)$, we can construct functions $\tilde \psi_{(l)}$, $\tilde \psi_{l}$ and $\tilde \psi$ in the same way.
\end{rema}

The following lemma shows that $Z$, $\phi$ are differentiable in the sense of Whitney on the support of $\psi$.
We follow the inductive step to find their extensions.

\begin{lemma} \label{whitney diff}
$Z$ and $\phi$ are $C^{r-2}$ in the sense of Whitney with respect to $E$ on $(\cup_{k}J_k) \cap supp(\psi)$. Moreover, there exist $c_3$, $D_2 > 0$ such that
$$|\partial_{E}^{j}(Z - Z_l)|_{(\cup_k J_k) \cap supp(\psi) }, \
|\partial_{E}^{j}(\phi -\phi_l)|_{(\cup_k J_k) \cap supp(\psi)} <c_3 e^{-\overline{Q}_l^{D_2}}, \quad j=0,\cdots, r-2$$
\end{lemma}
\proof Fix any $\theta\in\T$. We only detail the proof for $Z$, with that of $\phi$ being similar.

 Let $F_1 = Z_{1}$. By \eqref{properties}, there exists a sequence of $C^{\infty}$ functions on $\R$, denoted by $\{\Psi_l\}_l$,
with $0 \leq \Psi_l \leq 1$ and $|\Psi_l|_{C^{r-1}(\R)} \lesssim Q_l^{(r-1)\tau_0}$ for each $l \geq 1$, such that
\begin{equation}\label{inclution_supp}
supp(\Psi_{l+1}) \subset supp(\psi_{l})\cap \overline{\Omega}_{l+1} \subset \Psi_l^{-1}(1).
\end{equation}
For $l\geq 2$, we define $F_l := \Psi_l Z_{l} + (1-\Psi_l) F_{l-1}$. Hence,
$F_{l+1} - F_{l} = \Psi_{l+1}(Z_{l+1} - F_{l})$,
and $F_l (E) = Z_l(E)$ for all $E \in \Psi_l^{-1}(1)$. So $F_{l+1} - F_{l}= \Psi_{l+1}(Z_{l+1} - Z_l)$ on $\Psi_l^{-1}(1)$.

Recall that $Z_l=\Lambda W_l$. On each $J_k$, $1\leq k \leq K$, by Proposition \ref{prop reducibility} and \eqref{inclution_supp},
$$|F_{l+1} - F_{l}|_{C^{r-1}(J_k\cap\overline{\Omega}_0)}\leq |\Lambda| |\Psi_{l+1}|_{C^{r-1}(\R)}|W_{l+1} -W_l|_{C^{r-1}(J_k\cap\overline{\Omega}_{l+1})} \lesssim
 c_1|\Lambda| Q_{l+1}^{(r-1)\tau_0} e^{-\overline{Q}_l^{D_1}}.$$
Thus $\{F_{l}\}$ is a Cauchy sequence in $C^{r-1}(J_k \cap \overline{\Omega}_0)$.
So there exists $F \in C^{r-2}((\cup_k J_k)\cap \overline{\Omega}_0)$ such that $\lim_{l \to \infty}|F-F_l|_{C^{r-2}((\cup_k J_k)  \cap \overline{\Omega}_0)} = 0$.
 Moreover, there exist constants $c_3$, $D_2 > 0$, independent of $l$, such that
 $$|\partial_{E}^{j}(F_l - F)|_{(\cup_{k}J_k) \cap \overline{\Omega}_0} < c_3 e^{-\overline{Q}_{l}^{D_2}}, \quad  j= 0,\cdots, r-2.$$
On the other hand, since $F_l = Z_l$ on $(\cup_{k}J_k) \cap \Psi_l^{-1}(1)$, combining with the fact that $supp(\psi_{l})\cap \overline\Omega_{l+1} \subset \Psi_l^{-1}(1)$, we conclude that $F_l$ converges to $Z$ in $C^0$ on $(\cup_k J_k)\cap supp(\psi)$, noting that $supp(\psi)\subset \overline\Omega$, which follows from Proposition \ref{prop_psi}. Thus $F = Z$ on $(\cup_k J_k)\cap supp(\psi)$.
 This concludes the proof since $F$ is $C^{r-2}$ over $(\cup_k J_k)\cap \overline{\Omega}_0$.
\qed


\subsection{Modified spectral transformation}


For each $1 \leq k \leq K$, we choose a non-negative real function $\chi_k \in C^{\infty}(\R)$ such that
$supp(\chi_k) =  cl(J_k)$, i.e. $\chi_k(E) > 0$ for any $E \in J_k$ and $\chi_k(E) = 0$ for any $E \notin J_k$.
Obviously, $\chi_k$ vanishes on $\partial J_k$.
We further require that $\partial_{E}^{i}\chi_{k}(E) = 0$ for $E \in \partial J_k$ for $i=1, 2, 3$.
We define $\chi \in C^{\infty}(\R)$ by $\chi = \sum_{k}\chi_k$.

Since $A_{(E,V)}(\theta) = Z(E,\theta + \alpha)R_{\phi(E,\theta)} Z(E,\theta)^{-1}$ on $((\cup_k J_k) \cap \overline{\Omega}) \times \T$, we can see that $L_{\theta}f=Ef$ with
$$f_n(E,\theta):= e^{{\rm i}\phi^{[n]}(E,\theta)}[Z_{21}(E,\theta+n\alpha)-{\rm i}Z_{22}(E,\theta+n\alpha)].$$
Indeed, by noting that  $\left(\begin{array}{c}
1 \\[1mm]
-{\rm i}
\end{array}
\right)$ is an eigenvector of $R_{\phi(E,\theta)}$ corresponding to the eigenvalue $e^{{\rm i}\phi(E,\theta)}$,
with $\left(\begin{array}{c}
f_{1} \\[1mm]
f_{0}
\end{array}
\right)=Z(E,\theta)\left(\begin{array}{c}
1\\[1mm]
-{\rm i}
\end{array}
\right)$, we get the generalized eigenvector
$$\left(\begin{array}{c}
f_{n+1} \\[1mm]
f_{n}
\end{array}
\right)=Z(E,\theta+n\alpha) \, R_{\phi^{[n]}(E,\theta)} \, Z(E,\theta)^{-1} \left(\begin{array}{c}
f_{1} \\[1mm]
f_{0}
\end{array}
\right)=e^{{\rm i}\phi^{[n]}(E,\theta)} \, Z(E,\theta+n\alpha) \left(\begin{array}{c}
1 \\[1mm]
-{\rm i}
\end{array}
\right).$$
Then, for every $n\in\Z$, let
$$\left(\begin{array}{c}
{\cal K}_n \\[1mm]
{\cal J}_n
\end{array}\right)
=\chi\left(\begin{array}{c}
\Im f_n \\[1mm]
\Re f_n
\end{array}\right)
=\left(\begin{array}{c}
\beta_n\sin \phi^{[n]}-\gamma_n\cos \phi^{[n]}  \\[1mm]
\beta_n\cos \phi^{[n]}+\gamma_n\sin \phi^{[n]}
\end{array}\right),
$$
where $\beta_n(E,\theta):=\chi(E)Z_{21}(E,\theta+n\alpha)$, $\gamma_n(E,\theta):=\chi(E)Z_{22}(E,\theta+n\alpha)$. Obviously,
${\cal K}=({\cal K}_n)_{n\in\Z}$ and ${\cal J}=({\cal J}_n)_{n\in\Z}$ are two generalized eigenvectors of $L$.
Moreover, for every $l\in \N_+$, we can define the approximated coefficients
$$\beta^{(l)}_n(E,\theta):=\chi(E)(Z_l)_{21}(E,\theta+n\alpha), \quad \gamma^{(l)}_n(E,\theta):=\chi(E)(Z_l)_{22}(E,\theta+n\alpha).$$

\begin{lemma} \label{lemma approx}
 For each $1\leq k \leq K$, $\beta_n$, $\gamma_n$ are $C^{r-2}$ in the sense of Whitney on $J_k \cap \overline\Omega$, and $\beta^{(l)}_n$, $\gamma^{(l)}_n$ are $C^\infty$ on $J_k\cap\overline{\Omega}_l$.
Moreover, there exists $c_4>0$ such that, for every $n$ and every $l,k$, we have
\begin{eqnarray}
|\beta_n|_{C^{r-2}(J_k \cap \overline\Omega)}, \ |\gamma_n|_{C^{r-2}(J_k \cap \overline\Omega)}, \ |\beta^{(l)}_n|_{C^r(J_k \cap \overline{\Omega}_l)}, \ |\gamma^{(l)}_n|_{C^r(J_k \cap \overline{\Omega}_l)}&\leq& c_4, \label{beta_gamma_1}  \\
|\partial_{E}^{j}(\beta_n-\beta^{(l)}_n)|_{J_k \cap supp(\psi)}, \ |\partial_{E}^{j}(\gamma_n-\gamma^{(l)}_n)|_{J_k \cap  supp(\psi)}&\leq&c_4e^{-\overline{Q}_l^{D_2}}, \  j = 0, \cdots, r-2 , \label{beta_gamma_2}
\end{eqnarray}
\end{lemma}

\proof The regularity of $\beta_n$, $\gamma_n$, $\beta_n^{(l)}$, $\gamma_n^{(l)}$ follows easily from their definitions.

Recall that $\beta_n=\chi Z_{21}(\theta+n\alpha)$, $\beta^{(l)}_n=\chi(Z_l)_{21}(\theta+n\alpha)$ with $Z=\Lambda W$, $Z_l=\Lambda W_l$. Then, by (\ref{prop_error_1})  in Proposition \ref{prop reducibility}, we have
$$
|\beta_n|_{C^{r-2}(J_k \cap \overline\Omega)}, \ |\beta^{(l)}_n|_{C^r(J_k \cap \overline{\Omega}_l)}\leq 6 c_1 |\chi |_{C^r(J_k)} |\Lambda|.
$$
and in view of Lemma \ref{whitney diff}, we have
$$|\partial_{E}^{j}(\beta_n-\beta^{(l)}_n)|_{J_k \cap  supp(\psi)}\leq 2 c_3 |\chi |_{C^{r-2}(J_k)} e^{-\overline{Q}_l^{D_2}}. \quad j=0, \cdots, r-2.$$
By choosing $c_4:=10 \max (c_1|\Lambda|, c_3)\cdot|\chi |_{C^r(J_k)}$, we get the estimates $(\ref{beta_gamma_1}),(\ref{beta_gamma_2})$ for $\beta_n$ and $\beta_n^{(l)}$.
The proof for $\gamma_n$ and $\gamma_n^{(l)}$ is similar.\qed


%

Let $d\mu :=\left(\begin{array}{cc}
\psi\, dE & 0 \\[1mm]
0 & \psi\, dE
\end{array}\right)$.
Recall the definition of ${\cal L}^2-$space given in (\ref{gen_L2}).
${\cal L}^2(d\mu)$ means the space of vectors
$G=(g_j)_{j=1,2}$, with $g_j$ functions of $E\in\R$ satisfying
$$
\|G\|_{{\cal L}^2(d\mu)}^2:= \int_\R (|g_1|^2+|g_2|^2)\,  \psi dE<\infty.
$$
For any $\theta \in \T_{ac}$, we define the modified spectral transformation for the operator $L_{\theta}$ by
$$\begin{array}{clcl}
    {\cal S}_{\theta}: & {\cal W}^2(\Z) & \rightarrow & {\cal L}^2(d\mu) \\[1mm]
     & (u_n)_{n\in\Z} & \mapsto &  \displaystyle \left(\begin{array}{c}
                                                         \sum_{n \in \Z}u_n {\cal K}_n(\cdot, \theta) \\[1mm]
                                                         \sum_{n \in \Z}u_n {\cal J}_n(\cdot, \theta)
                                                       \end{array}
     \right)
  \end{array}.$$
Recall that $\Sigma_0\subset\overline{\Omega}$.
Since ${\cal K}_n$, ${\cal J}_n$ are uniformly bounded for all $(E, \theta)\in ((\cup_k J_k) \cap \Sigma_0) \times  \T$ and $n \in \Z$, and $0\leq \psi \leq 1$, we can see by Cauchy inequality that for any $u\in{\cal W}^2(\Z)$,
$$\left\|\left(\begin{array}{c}
       \sum_{n \in \Z}u_n {\cal K}_n(\cdot, \theta) \\
       \sum_{n \in \Z}u_n {\cal J}_n(\cdot, \theta)
     \end{array}
\right)\right\|^2_{{\cal L}^2(d\mu)}\leq (\|{\cal K}\|_{\ell^\infty}^2+\|{\cal J}\|_{\ell^\infty}^2) \sum_{m,n\in\Z}|u_m| |u_n| \lesssim \lan u\ran_{2}^2.$$
For $M  > 0$, we also define the truncated spectral transformation
$$\begin{array}{clcl}
    {\cal S}_{\theta, M}: & \ell^2(\Z) & \rightarrow & {\cal L}^2(d\mu) \\[1mm]
     & (u_n)_{n\in\Z} & \mapsto &  \displaystyle \left(\begin{array}{c}
                                                         \sum_{|n| \leq M}u_n {\cal K}_n(\cdot, \theta) \\[1mm]
                                                         \sum_{|n| \leq M}u_n {\cal J}_n(\cdot, \theta)
                                                       \end{array}
     \right)
  \end{array}.$$
It is clear that given any $u\in {\cal W}^2(\Z)$, for every $E\in ((\cup_k J_k) \cap \overline{\Omega})$,
$$ \lim_{M\rightarrow \infty}\sum_{|n|> M}|u_n {\cal K}_n(E, \theta)|=\lim_{M\rightarrow \infty}\sum_{|n|> M}|u_n {\cal J}_n(E, \theta)|=0.$$
So ${\cal S}_{\theta}u$ is the pointwise limit of ${\cal S}_{\theta, M}u$ as $M\rightarrow\infty$.

For given $\theta\in\T_{ac}$, we choose $\varepsilon > 0$ and $\psi$ accordingly, so that by \eqref{supp positive} we have
\begin{eqnarray*}
\quad \left| \left\{E \in (\cup_k J_k) \cap \Sigma_0 \left| \psi(E) > 0, \,  \left(\begin{array}{c}
                                      \sum_{n} u_n(0){\cal K}_n(E, \theta)\\
                                       \sum_{n} u_n(0){\cal J}_n(E, \theta)
                                    \end{array}
\right) \neq 0\right. \right\}\right| > \frac{e_0}{2}.
\end{eqnarray*}
Then, as the modified spectral transformation with measure $d\mu$, $S_{\theta}$ has the property that $S_{\theta}u \neq 0$. Thus by dominated convergence theorem, we have
\begin{eqnarray} \label{positive area}
\lim_{M \to \infty} \norm{S_{\theta, M}u}_{{\cal L}^2(d\mu)}  = \norm{ S_{\theta}u }_{{\cal L}^2(d\mu)} > 0.
\end{eqnarray}

\begin{rema}
As shown in Subsection \ref{pre_op_cocycle},
the classical spectral transformation is a unitary transformation from $\ell^2(\Z)$ to ${\cal L}^2(d\mu)$, with $d\mu$ the matrix of spectral measures introduced by $m-$functions.
In contrast, to get better differentiability with respect to $E$, the modified spectral transformation ${\cal S}_\theta$ here is not a unitary one, and we define it just on the subspace ${\cal W}^2(\Z)$.
Comparing with (\ref{classical_uv}) for the free Schr\"odinger operator, ${\cal K}_n$ and ${\cal J}_n$ for ${\cal S}_\theta$ have no divisor as `` $\sim\sin\xi_0$" and they have a smoothing factor $\chi$ which covers the singularities.
Moreover, instead of the spectral measures shown in Theorem \ref{spectral_measure_matrix},
we use the explicit measure $\psi\, dE$, which serves also as a smoothing factor.
\end{rema}

%
%

%
%
%
%

By direct computations, we obtain
\begin{equation}\label{derivative_KnJn}
\left(\begin{array}{c}
\partial_E{\cal K}_n \\[1mm]
\partial_E{\cal J}_n
\end{array}\right)=\left(\begin{array}{c}
{\cal B}_n\cos\phi^{[n]} + {\cal C}_n\sin\phi^{[n]} \\[1mm]
- {\cal B}_n\sin\phi^{[n]} + {\cal C}_n\cos\phi^{[n]}
\end{array}\right)
\end{equation}
with ${\cal B}_n:=\beta_n \partial_{E}\phi^{[n]} - \partial_{E}\gamma_n$, ${\cal C}_n:=\gamma_n \partial_{E}\phi^{[n]} + \partial_{E}\beta_n$, where $\partial_E$ is interpreted in the sense of Whitney.
Similarly, we define the approximated coefficients
${\cal B}^{(l)}_n:=\beta^{(l)}_n \partial_{E}\phi_l^{[n]} - \partial_{E}\gamma^{(l)}_n$,
${\cal C}^{(l)}_n:=\gamma^{(l)}_n \partial_{E}\phi_l^{[n]} + \partial_{E}\beta^{(l)}_n$. According to Lemma \ref{whitney diff} and \ref{lemma approx}, $\beta_n$, $\gamma_n$ are bounded on $J_k\cap supp(\psi)$, and $\beta^{(l)}_n$, $\gamma^{(l)}_n$ are $C^\infty$ on $J_k\cap\overline\Omega_{l}$.
Combining with the fact that
$|\partial_E \phi_{l}^{[n]}| \leq c_1 n$, we can find a constant $c_5>0$, such that
\begin{eqnarray}
|{\cal B}_n|_{J_k \cap supp(\psi)},
\ |{\cal C}_n|_{J_k \cap supp(\psi)}, \ |{\cal B}^{(l)}_n|_{C^2(J_k \cap \overline{\Omega}_l)},
\ |{\cal C}^{(l)}_n|_{C^2(J_k \cap \overline{\Omega}_l)} &\leq&  c_5 n,\label{partial_KJ_1}\\
|{\cal B}_n-{\cal B}^{(l)}_n|_{J_k \cap supp(\psi)}, \ |{\cal C}_n-{\cal C}^{(l)}_n|_{J_k \cap  supp(\psi)}   &\leq& c_5 n e^{-\overline{Q}_l^{D_2}}, \label{partial_KJ_2}\\
|\partial_E{\cal K}_n|_{J_k \cap  supp(\psi)}, \ |\partial_E{\cal J}_n|_{J_k \cap  supp(\psi)}&\leq& c_5 n.\label{partial_KJ_4}
\end{eqnarray}

\

To consider the ballistic lower bound for any $p\geq 0$, it can be reduced to considering the case that $p$ is an even integer.
This will be discussed in Subsection \ref{The general case}.
Assume that $p \geq 2$ is an even integer, and let $r =p/2 + 2$.
We have
\begin{equation}\label{derivative_KnJn2}
\left(\begin{array}{c}
\partial_E^{r-2}{\cal K}_n \\[1mm]
\partial_E^{r-2}{\cal J}_n
\end{array}\right)=\left(\begin{array}{c}
{\cal B}_{n,r}\cos\phi^{[n]} + {\cal C}_{n,r}\sin\phi^{[n]} \\[1mm]
- {\cal B}_{n,r}\sin\phi^{[n]} + {\cal C}_{n,r}\cos\phi^{[n]}
\end{array}\right)
\end{equation}
with $\cal B_{n,r}$, $\cal C_{n,r} $ being a linear combination of monomials of $\partial_{E}^{j}\beta_{n}$, $\partial_E^{j}\gamma_n$, $\partial_E^{j} \phi^{[n]}$ with sum of the degrees of derivatives no more than $r-2$. We define the approximants $\cal B_{n,r}^{(l)}$, $\cal C_{n,r}^{(l)} $ in the similar way.  By Lemma \ref{whitney diff}, \ref{lemma approx} and  \eqref{prop_error_2}, there exists $c_6 > 0$ such that
\begin{eqnarray}
|{\cal B}_{n,r}|_{J_k \cap supp(\psi)},
\ |{\cal C}_{n,r}|_{J_k \cap supp(\psi)}, \ |{\cal B}^{(l)}_{n,r}|_{C^2(J_k \cap \overline{\Omega}_l)},
\ |{\cal C}^{(l)}_{n,r}|_{C^2(J_k \cap \overline{\Omega}_l)} &\leq&  c_6 n^{r-2},\label{2partial_KJ_1}\\
|{\cal B}_{n,r}-{\cal B}^{(l)}_{n,r}|_{J_k \cap supp(\psi)}, \ |{\cal C}_{n,r}-{\cal C}^{(l)}_{n,r}|_{J_k \cap  supp(\psi)}   &\leq& c_6 n^{r-2} e^{-\overline{Q}_l^{D_2}}, \label{2partial_KJ_2}\\
|\partial_E^{r-2}{\cal K}_{n}|_{J_k \cap  supp(\psi)}, \ |\partial_E^{r-2}{\cal J}_{n}|_{J_k \cap  supp(\psi)}&\leq& c_6 n^{r-2}.\label{2partial_KJ_4}
\end{eqnarray}

\subsection{Ballistic lower bound for $p=2$}

In this subsection, we will prove Theorem \ref{main theorem 2} for $p=2$. This corresponds to the estimates for $r=3$.  The proof for the general case for arbitrary $p$ is completely analogy to that of $p=2$, we will sketch the needed modifications in the next subsection.

From now on, we fix $\theta \in \T_{ac}$, and we will not express this dependence explicitly.
With $u(t)$ the solution of Eq. (\ref{Eq_Schrodinger}) with initial condition $u(0) = u \in {\cal W}^{2}(\Z) \setminus \{0\}$, we consider
\begin{eqnarray*}
&&G(t, E)= \left(G^{(j)}(t,E)\right)_{j=1,2} := S_{\theta}(u(t)),  \\
&&G_{T}(t,E)= \left(G_{T}^{(j)}(t,E)\right)_{j=1,2}:= S_{\theta, \floor*{T^{100}}+1}(u(t)) \ {\rm for \ any } \ T>0 .
\end{eqnarray*}
According to Eq. (\ref{Eq_Schrodinger}),  we have
\begin{eqnarray*}
{\rm i}\partial_{t}G_{T}(t,E) &=& {\rm i} \left(\begin{array}{c}
                    \sum_{|n| \leq \floor*{ T^{100}}+1} \partial_{t}u_{n}(t) {\cal K}_{n}(E, \theta) \\[1mm]
                    \sum_{|n| \leq  \floor*{T^{100}}+1} \partial_{t}u_{n}(t) {\cal J}_{n}(E, \theta)
                  \end{array}
\right) \\[1mm]
&=& \left(\begin{array}{c}
                    \sum_{|n| \leq \floor*{T^{100}}+1} (Lu)_{n}(t) {\cal K}_{n}(E, \theta) \\[1mm]
                    \sum_{|n| \leq \floor*{T^{100}}+1} (Lu)_{n}(t) {\cal J}_{n}(E, \theta)
                  \end{array}
\right) \\[1mm]
&=& \left(\begin{array}{c}
                    \sum_{|n| \leq \floor*{T^{100}}+1} u_{n}(t) (L{\cal K})_{n}(E, \theta) \\[1mm]
                    \sum_{|n| \leq \floor*{T^{100}}+1} u_{n}(t) (L{\cal J})_{n}(E, \theta)
                  \end{array}
\right)+F_{T}(t,E) \\[1mm]
&=& EG_T(t,E) + F_{T}(t,E),
\end{eqnarray*}
where $F_{T}(t,E)=\left(F^{(j)}_{T}(t,E)\right)_{j=1,2}$ is given by
\begin{eqnarray*}
F_{T}(t,E)
&= & \left(\begin{array}{c}
u_{\floor*{T^{100}}+1}(t){\cal K}_{\floor*{T^{100}}+2}(E, \theta)-u_{\floor*{T^{100}}+2}(t){\cal K}_{\floor*{T^{100}}+1}(E, \theta)\\[1mm]
u_{\floor*{T^{100}}+1}(t){\cal J}_{\floor*{T^{100}}+2}(E, \theta)-u_{\floor*{T^{100}}+2}(t){\cal J}_{\floor*{T^{100}}+1}(E, \theta)
\end{array}
\right) \nonumber \\[1mm]
& & +\left(\begin{array}{c}
u_{-\floor*{T^{100}}-1}(t){\cal K}_{-\floor*{T^{100}}-2}(E, \theta)-u_{-\floor*{T^{100}}-2}(t){\cal K}_{-\floor*{T^{100}}-1}(E, \theta)\\[1mm]
u_{-\floor*{T^{100}}-1}(t){\cal J}_{-\floor*{T^{100}}-2}(E, \theta)-u_{-\floor*{T^{100}}-2}(t){\cal J}_{-\floor*{T^{100}}-1}(E, \theta)
\end{array}
\right). \label{F T}
\end{eqnarray*}
Hence we have $G^{(j)}_{T}(t,E) = e^{-itE}G^{(j)}_{T}(0,E) -i \int_{0}^{t} e^{-i(t-s)E}F^{(j)}_{T}(s,E) ds$, and
\begin{eqnarray}
\partial_{E}G^{(j)}_{T}(t,E)&=&(-it)e^{-itE}G^{(j)}_{T}(0,E) + e^{-itE}\partial_EG^{(j)}_{T}(0,E)  \nonumber \\
   & & -  \int_{0}^{t}(t-s)e^{-i(t-s)E}F^{(j)}_{T}(s,E) ds -i \int_{0}^{t} e^{-i(t-s)E} \partial_EF^{(j)}_{T}(s,E) ds.  \label{partial G T}
\end{eqnarray}

The following two lemmas show that $F_{T}$ is negligible up to time $T$.

\begin{lemma} \label{F T small}
There is a constant $c_7 > 0$ such that, given any $T>1$, for $0 \leq t \leq T$,
\begin{equation}\label{F T L2}
\|F_{T}(t,E)\|^2_{{\cal L}^2(d\mu)} \leq  c_7 T^{-3}.
\end{equation}
\end{lemma}
\proof
By the expression of $F_{T}$, we have, for $j=1,2$,
$$|F^{(j)}_{T}(t,E)| \leq (\|\cal K\|_{\ell^\infty} + \|\cal J\|_{\ell^\infty})\sum_{i=1,2}\sum_{N=\pm\left(\floor*{T^{100}}+i\right)}|u_{N}(t)|.$$
For $N  = \floor*{T^{100}}+1$, $\floor*{T^{100}}+2$, $u_N(t) = \langle e^{-itL}u(0), \delta_N \rangle = \langle u(0), e^{itL} \delta_N \rangle$.
Thus for any $0 \leq t \leq T$,
$$|u_N(t)| \leq \|u(0)\|_{\ell^2(\Z)} \left(\sum_{n<T^{10}} |(e^{itL} \delta_N)_n|^2 \right)^{\frac12}+ \|e^{itL}\delta_N\|_{\ell^2(\Z)} \left( \sum_{n\geq T^{10}}|u_n(0)|^2 \right)^{\frac12}.$$
By Cauchy inequality,
$$ \sum_{n\geq T^{10}}|u_n(0)|^2\leq\left(\sum_{n\geq T^{10}}\frac{1}{n^2}\right)\langle u(0)\rangle^2_{2} \lesssim T^{-10}\langle u(0)\rangle^2_{2}.$$
Let $S:\ell^2(\Z)\to\ell^2(\Z)$ be the unitary operator defined as $(Sq)_n=q_{n-1}$.
Thus
\begin{eqnarray}
\sum_{n<T^{10}} |(e^{itL} \delta_N)_n|^2&=&\sum_{n+N<T^{10}}|(e^{itL} \delta_N)_{n+N}|^2\nonumber\\
&=&\sum_{n<T^{10}-N}|(S^{-N}e^{itL}S^{N} \delta_0)_{n}|^2\nonumber\\
&\leq& \left(\sum_{n<T^{10}-N}\frac{1}{n^2}\right) \left(\sum_{n\in\Z}n^2|(S^{-N}e^{itL}S^{N} \delta_0)_{n}|^2\right).\label{F_T_part1}
\end{eqnarray}
By the fact that $N\gg 2T^{10}$, we see
$$\sum_{n<T^{10}-N}\frac{1}{n^2}< \sum_{n<-T^{10}}\frac{1}{n^2}\lesssim T^{-10}.$$
Note that $S^{-N}e^{itL}S^{N}$ equals to $e^{it(S^{-N}LS^{N})}$ and the operator $S^{-N}LS^{N}:\ell^2(\Z)\to\ell^2(\Z)$ satisfies
$$(S^{-N}LS^{N} q)_n=-(q_{n+1}+q_{n-1})+V(\theta+n\alpha+N\alpha) q_n.$$
Then, applying Theorem 2.22 in \cite{DamanikTch} with $p=2$ and $\eta=\frac32$, we get, for any $0\leq t \leq T$,
$$\sum_{n<T^{10}} |(e^{itL} \delta_N)_n|^2\lesssim T^{-10}\cdot (t^{3}+1) \lesssim T^{-6}. $$
Hence $|u_N(t)| \lesssim \|u(0)\|_{\ell^2(\Z)} T^{-3}+ \lan u(0) \ran_2 T^{-5}$ for $N  = \floor*{T^{100}}+1$, $\floor*{T^{100}}+2$.
In the same way, we can get the same estimate for $N  = -\floor*{T^{100}}-1$, $-\floor*{T^{100}}-2$.
So there exists a constant $c_7 = c_7(\lan u(0)\ran_{2},\|u(0)\|_{\ell^2(\Z)}, \|{\cal K}\|_{\ell^\infty})$ such that
 for $j=1,2$,
$$|F^{(j)}_T(t,E)| \leq  \frac{\sqrt{c_7}}{4} T^{-3}, \quad \forall 0\leq t \leq T.$$
This proves \eqref{F T L2}.\qed

%

\begin{lemma}\label{P_F T small}We have $\lim_{T \to \infty} \sup_{0 \leq t \leq T}  \|\partial_E F_{T}(t,E)\|^2_{{\cal L}^2(d\mu)} = 0$.
\end{lemma}
\proof
In view of the expression of $\partial_E F_{T}$, combining with (\ref{partial_KJ_4}), we have, for $j=1,2$,
$$|\partial_E F^{(j)}_{T}(t,E)| \lesssim \sum_{i=1,2}\sum_{N=\pm\left(\floor*{T^{100}}+i\right)}|N u_{N}(t)|.$$
By proceeding in a similar fashion as in the previous lemma, we obtain, for $N  = \floor*{T^{100}}+1$, $\floor*{T^{100}}+2$ and $0 < t \leq T$,
\begin{eqnarray*}
|Nu_N(t)|   &\leq& N \|u(0)\|_{\ell^2(\Z)} \left(\sum_{n<\frac{N}{2}} |(e^{itL} \delta_N)_n|^2 \right)^{\frac12}+ N \|e^{itL}\delta_N\|_{\ell^2(\Z)} \left( \sum_{n\geq \frac{N}{2}}|u_n(0)|^2 \right)^{\frac12} \\
   &\leq&  2\|u(0)\|_{\ell^2(\Z)} \left(\sum_{n<\frac{N}{2}} (n-N)^2|(e^{itL} \delta_N)_n|^2 \right)^{\frac12}+ 2 \left( \sum_{n\geq \frac{N}{2}}n^2|u_n(0)|^2 \right)^{\frac12}.
\end{eqnarray*}
Since $u(0) \in {\cal W}^{2}(\Z)$,  we have
\begin{equation}\label{partial_F_T_part_2}
\lim_{N\rightarrow\infty} \left( \sum_{n\geq \frac{N}{2}}n^2|u_n(0)|^2 \right)^{\frac12}= 0.
\end{equation}
Similar to (\ref{F_T_part1}), we have
\begin{eqnarray*}
\sum_{n<\frac{N}{2}} (n-N)^2 |(e^{itL} \delta_N)_n|^2&=&\sum_{n+N<\frac{N}{2}}n^2|(e^{itL} \delta_N)_{n+N}|^2\\
&=&\sum_{n<-\frac{N}{2}}n^2|(S^{-N}e^{itL}S^{N} \delta_0)_{n}|^2\\
&\leq& \left(\sum_{n<-\frac{N}{2}}\frac{1}{n^2}\right) \left(\sum_{n\in\Z}n^4|(S^{-N}e^{itL}S^{N} \delta_0)_{n}|^2\right).
\end{eqnarray*}
Applying Theorem 2.22 in \cite{DamanikTch} with $p=4$ and $\eta=\frac54$, we get
$$\sum_{n<\frac{N}{2}} (n-N)^2 |(e^{itL} \delta_N)_n|^2\lesssim N^{-2} (t^5+1)\ll T^{-5}.$$
Combining with (\ref{partial_F_T_part_2}), we finish the proof.\qed

\smallskip

For any $n \in \Z$, let $\lan n \ran:=\sqrt{n^2+1}$ and for $k\in\Z$, we denote
\begin{equation}\label{a k}
a_{k} := \sup_{n \in \Z} \frac{1}{\lan n+k \ran \lan n \ran}\left|\int_\R [(\partial_{E}{\cal K}_{n+k})(\partial_{E} {\cal K}_n) + (\partial_{E}{\cal J}_{n+k}) (\partial_{E}{\cal J}_n)] \psi \, dE \right|.
\end{equation}
By (\ref{partial_KJ_4}), we can easily see that $a_k< 10 c^2_5$ for every $k\in\Z$.
The following lemma relates the lower bound of $\langle u(t) \rangle_{2}$ to $\{a_k\}$.

\begin{lemma} \label{norm growth}
There exists $c_8 > 0$, such that for sufficiently large $T > 0$, we have
\begin{eqnarray*}
 \langle u(T) \rangle_{2} \geq \frac{c_8 T}{(T^{\eta} +  \sum_{T^{\eta} < |k| \leq T^{101}} a_k)^{\frac{1}{2}}} - c_8^{-1}
\end{eqnarray*}
\end{lemma}
\proof By direct calculations, we can see
\begin{eqnarray*}
&&\|\partial_{E}G_{T}(T,\cdot)\|^2_{{\cal L}^2(d\mu)}\\
&=& \sum_{|m|,|n| \leq \floor*{ T^{100}}+1} \lan m\ran \lan n \ran u_m(T) {\overline u}_n(T) \int_\R \frac{1}{ \lan m\ran \lan n \ran}[(\partial_{E}{\cal K}_{m}) (\partial_{E} {\cal K}_n) + (\partial_{E}{\cal J}_m) (\partial_{E}{\cal J}_n)]\psi \, dE\\
&\leq&\sum_{|m|,|n| \leq \floor*{ T^{100}}+1}  \lan m\ran \lan n \ran |u_m(T)| |{\overline u}_n(T)| \  a_{m-n}.
\end{eqnarray*}
We can decompose $\|\partial_{E}G_{T}(T,\cdot)\|^2_{{\cal L}^2(d\mu)}$ into $I_1 + I_2$, corresponding to the summations $\sum_{|m|,|n| \leq \floor*{ T^{100}}+1\atop{|m-n| > T^{\eta}}}$ and $\sum_{|m|,|n| \leq \floor*{ T^{100}}+1\atop{|m-n| \leq T^{\eta}}}$ respectively. Since $a_k< 10 c^2_5$ for every $k\in\Z$, we have
\begin{equation}\label{I_2}
|I_2| \leq 10 c^2_5 \sum_{|k|\leq T^{\eta}}\sum_{n \in \Z} \lan n+k\ran \lan n\ran |u_{n+k}(T)| |\bar u_n(T)|
<40 c^2_5 T^{\eta}\langle u(T) \rangle_{2}^{2},
\end{equation}
and in view of Cauchy inequality, we obtain
\begin{equation}\label{I_1}
|I_1| \leq \sum_{T^{\eta} < |k| \leq T^{101}} a_k \langle u(T) \rangle_{2}^2.
\end{equation}
Then, by combining (\ref{I_2}) and (\ref{I_1}),
\begin{equation}\label{last term}
\left(40 c^2_5 T^{\eta} +  \sum_{T^{\eta} < |k| \leq T^{101}} a_k\right) \langle u(T) \rangle_{2}^2 \geq \|\partial_{E}G_{T}(T,E)\|_{{\cal L}^2(d\mu)}^{2}.
\end{equation}
Similarly, we can see that
\begin{equation} \label{last term 2}
\left(40 c^2_5 T^{\eta} +  \sum_{T^{\eta} < |k| \leq T^{101}} a_k\right) \langle u(0) \rangle_{2}^2 \geq \|\partial_{E}G_{T}(0,E)\|_{{\cal L}^2(d\mu)}^{2}.
\end{equation}


In view of (\ref{positive area}), we have that
$\lim_{T \to \infty}\|G_{T}(0,E)\|_{{\cal L}^2(d\mu)} = \|G(0,E)\|_{{\cal L}^2(d\mu)} > 0$.
So, combining (\ref{partial G T}), \eqref{last term 2}, Lemma \ref{F T small} and \ref{P_F T small},
we can find a constant $c'_8> 0$ such that for sufficiently large $T > 0$,
\begin{eqnarray*}
\|\partial_{E}G_{T}(T,E)\|_{{\cal L}^2(d\mu)} \geq c'_8 T  - \left(40 c^2_4 T^{\eta} +  \sum_{T^{\eta} < |k| \leq T^{101}} a_k\right)^{\frac12} \langle u(0) \rangle_{2}.
\end{eqnarray*}
The lemma follows together with (\ref{last term}).\qed

\smallskip

Given $l\in\Z_+$ large enough, let  ${\cal M}_{l}:=\left[Q_{l}^{8\tau_0}, \,  \min\left(Q_{l+1}^{\frac{\tau_1}{16}}, \overline{Q}_{l+1}^{\frac{\nu}{16}}\right)\right]$.  The following lemma shows that $\{a_k\}_{k}$ is summable on the sequence of intervals $\{{\cal M}_{l}\}$.

\begin{lemma}\label{hatKJ}
 For sufficiently large $l$, for any $k \in \Z\setminus\{0\}$ such that $|k| \in{\cal M}_{l}$, we have $a_k \leq |k|^{-\frac32}$.
\end{lemma}

\proof Recall the expressions of $\partial_E {\cal K}_n$ and $\partial_E {\cal J}_n$ in (\ref{derivative_KnJn}). Then for fixed $m,n\in\Z\setminus\{0\}$ with $ |m-n|\in {\cal M}_{l}$, we have $\displaystyle \frac{1}{\lan m\ran \lan n\ran}\int_\R [(\partial_E {\cal K}_m)(\partial_E {\cal K}_n)+(\partial_E {\cal J}_m)(\partial_E {\cal J}_n)] \psi \, dE ={\cal P}+{\cal T}$, with
\begin{eqnarray*}
{\cal P}&:=&\frac{1}{\lan m\ran \lan n\ran}\int_\R ({\cal B}_m {\cal B}_n + {\cal C}_m {\cal C}_n)\cos(\phi^{[m]}-\phi^{[n]})\cdot\psi\, dE, \\
{\cal T}&:=&\frac{1}{\lan m\ran \lan n\ran}\int_\R ({\cal B}_m {\cal C}_n - {\cal C}_m {\cal B}_n)\sin(\phi^{[m]}-\phi^{[n]})\cdot\psi\, dE.
\end{eqnarray*}
We are going to show that $|{\cal P}|$, $|{\cal T}| \leq \frac12 |m-n|^{-\frac32}$. We will detail the estimate of ${\cal P}$, and one can estimate ${\cal T}$ in a similar way.

Noting that $\phi^{[m]}(\theta)-\phi^{[n]}(\theta)=\phi^{[m-n]}(\theta+n\alpha)$,
we have
$${\cal P} = \int_\R h(E,\theta) \psi(E) \cos \phi^{[m-n]}(E,\theta+n\alpha)dE,$$
with $\displaystyle h:= \frac1{\lan m\ran \lan n\ran}({\cal B}_m {\cal B}_n + {\cal C}_m {\cal C}_n)$. We can consider the approximated integral
$${\cal P}_l = \int_\R h_l(E,\theta)\psi_{l}(E) \cos \phi_l^{[m-n]}(E,\theta+n\alpha) dE$$
instead, where
$\displaystyle h_l := \frac1{\lan m\ran \lan n\ran}({\cal B}_m^{(l)} {\cal B}_n^{(l)} + {\cal C}_m^{(l)} {\cal C}_n^{(l)})$. Indeed, in view of $(\ref{partial_KJ_1}),  (\ref{partial_KJ_2})$, we have $|h_l - h|_{J_k\cap  supp(\psi)} \leq 4 c_5^2 e^{-\overline{Q}_l^{D_2}}$ and
\begin{equation}\label{h_l h 1}
|h|_{J_k \cap  supp(\psi)}, \ |h_l|_{C^2(J_k \cap \overline{\Omega}_l)} \leq 2c_5^2.
\end{equation}
By Lemma \ref{whitney diff}, $|\phi_l^{[m-n]} - \phi^{[m-n]}|_{J_k\cap \overline{\Omega}}< c_2 |m-n| e^{-\overline{Q}_l^{D_2}}$.
Since $|m-n|\in{\cal M}_l$ with ${\cal M}_l$ defined by the $({\cA}, {\cA}, {\cA}^{22}, {\cA}^{21})-$admissible subsequence $(Q_l)$, we can see that,
 for large $l$,
$$|m-n|e^{-\overline{Q}_l^{D_2}}\leq Q_{l+1}^{\frac{\tau_1}{16}}e^{-Q_{l+1}^{D_2 {\cal A}^{-21}}} \leq Q_{l+1}^{-\tau_1}\leq |m-n|^{-5}.$$
And then
 $$\left|h_l(E,\theta)\cos\phi_l^{[m-n]}(E,\theta+n\alpha) - h(E,\theta)\cos\phi^{[m-n]}(E,\theta+n\alpha)\right|_{J_k\cap supp(\psi)}<|m-n|^{-4}.$$
Combing with Proposition \ref{prop_psi}, we can see that
$$|{\cal P}-{\cal P}_l| \leq 5 |m-n|^{-4} + 2 c_5^2 |\psi-\psi_l|_{L^1(\R)} \leq |m-n|^{-2}. $$

To compute ${\cal P}_l$, we apply the integration by parts on each $J_k$.
Recalling the construction of $C^\infty$ function $\chi$,
we have $h_l=\partial_E h_l=0$ at both boundary points of $J_k$.
So we have
\begin{eqnarray}
&&\int_{J_k} h_l(E,\theta) \psi_{l}(E) \cos \phi_l^{[m-n]}(E,\theta+n\alpha) dE\nonumber\\
&=& -\int_{J_k} \partial_E\left( \frac{h_l(E,\theta) \psi_{l}(E)}{\partial_E \phi_l^{[m-n]}(E,\theta+n\alpha)}\right ) \sin \phi_l^{[m-n]}(E,\theta+n\alpha) dE\nonumber\\
&=&-\int_{J_k}  \frac{\partial_E (h_l(E,\theta) \psi_{l}(E))}{\partial_E \phi_l^{[m-n]}(E,\theta+n\alpha)}  \sin \phi_l^{[m-n]}(E,\theta+n\alpha) dE\label{int_by_parts_01}\\
&& +\int_{J_k}  \frac{h_l(E,\theta) \psi_{l}(E)\cdot \partial_E^2 \phi_l^{[m-n]}(E,\theta+n\alpha)}{(\partial_E \phi_l^{[m-n]}(E,\theta+n\alpha))^2} \sin \phi_l^{[m-n]}(E,\theta+n\alpha) dE.\label{int_by_parts_02}
\end{eqnarray}
For the integrals in (\ref{int_by_parts_01}) and (\ref{int_by_parts_02}), by applying again the integration by parts, and noting that $\partial_E h_l=0$ at each edge point of $J_k$, we have
\begin{eqnarray*}
&&\int_{J_k}  \frac{\partial_E (h_l(E,\theta) \psi_{l}(E))}{\partial_E \phi_l^{[m-n]}(E,\theta+n\alpha)}  \sin \phi_l^{[m-n]}(E,\theta+n\alpha) dE\\
&=& \int_{J_k} \partial_E\left( \frac{\partial_E (h_l(E,\theta) \psi_{l}(E))}{(\partial_E \phi_l^{[m-n]}(E,\theta+n\alpha))^2}\right ) \cos \phi_l^{[m-n]}(E,\theta+n\alpha) dE,\\
&&\int_{J_k}  \frac{h_l(E,\theta) \psi_{l}(E)\cdot \partial_E^2 \phi_l^{[m-n]}(E,\theta+n\alpha)}{(\partial_E \phi_l^{[m-n]}(E,\theta+n\alpha))^2} \sin \phi_l^{[m-n]}(E,\theta+n\alpha) dE\\
&=&-\int_{J_k} \partial_E\left( \frac{h_l(E,\theta) \psi_{l}(E)\cdot \partial_E^2 \phi_l^{[m-n]}(E,\theta+n\alpha)}{(\partial_E \phi_l^{[m-n]}(E,\theta+n\alpha))^3}\right ) \cos \phi_l^{[m-n]}(E,\theta+n\alpha) dE.
\end{eqnarray*}
By direct calculations using that $|\psi_{l}|_{C^2(\overline{\Omega}_l)}\lesssim l^2 Q_l^{\tau_1}$, and the estimates of $h_l$(see (\ref{h_l h 1})), $\phi_l^{[m-n]}$(see (\ref{partial_phi_1}) and (\ref{partial_phi_2})), we can find a constant $c_9>0$ such that
\begin{eqnarray*}
&& \left|\partial_E\left( \frac{\partial_E (h_l(E,\theta) \psi_{l}(E))}{(\partial_E \phi_l^{[m-n]}(E,\theta+n\alpha))^2}\right )\right| \\
&=&\left|\frac{\partial^2_E (h_l(E,\theta) \psi_{l}(E))}{(\partial_E \phi_l^{[m-n]}(E,\theta+n\alpha))^2}
-\frac{2\partial_E (h_l(E,\theta) \psi_{l}(E))\cdot (\partial^2_E \phi_l^{[m-n]}(E,\theta+n\alpha))}{(\partial_E \phi_l^{[m-n]}(E,\theta+n\alpha))^3}\right|\\
&\leq& c_9 l^2 Q_{l}^{\tau_1}|m-n|^{-2},\\[1mm]
 &&\left|\partial_E\left( \frac{h_l(E,\theta) \psi_{l}(E)\cdot \partial_E^2 \phi_l^{[m-n]}(E,\theta+n\alpha)}{(\partial_E \phi_l^{[m-n]}(E,\theta+n\alpha))^3}\right)\right| \\
 &=&\left|\frac{\partial_E(h_l(E,\theta) \psi_{l}(E))\cdot \partial_E^2 \phi_l^{[m-n]}(E,\theta+n\alpha)}{(\partial_E \phi_l^{[m-n]}(E,\theta+n\alpha))^3}+\frac{h_l(E,\theta) \psi_{l}(E)\cdot \partial_E^3 \phi_l^{[m-n]}(E,\theta+n\alpha)}{(\partial_E \phi_l^{[m-n]}(E,\theta+n\alpha))^3}\right. \\
 & &\left.-\frac{3 h_l(E,\theta) \psi_{l}(E)\cdot (\partial^2_E \phi_l^{[m-n]}(E,\theta+n\alpha))^2}{(\partial_E \phi_l^{[m-n]}(E,\theta+n\alpha))^4}\right|\\
&\leq& c_9 l^2 Q_{l}^{\tau_1}|m-n|^{-2} .
\end{eqnarray*}
By $\tau_0 \geq \tau_1$ and $|m-n| \in \cal{M}_l$, both expressions above can be bounded by $\frac{1}{10}|m-n|^{-\frac32}$ for sufficiently large $l$.
This concludes that $|{\cal P}|\leq \frac12|m-n|^{-\frac32}$.\qed

 As a result of Lemma \ref{hatKJ}, we have
  $\sum_l \sum_{|k| \in {\cal M}_{l}}  a_k <\infty$.
Combined with Lemma \ref{norm growth}, we immediately obtain the following.
\begin{cor} \label{cor 1}
There exists $c_{10} > 0$, such that for any sufficiently large $l \geq 0$, for any $T > 0$ such that $[T^{\eta}, T^{101}] \subset  \cal M_l$, we have
$ \langle u(T)  \rangle_2 \geq c_{10} T^{1-\eta}$.
\end{cor}

In a similar way, with $(Q_l)$ replaced by $(R_l)$, we can use $\tilde{Z}$, $\tilde{\phi}$ and $\tilde\psi$(see Remark \ref{rmq_tilde} and \ref{rmq_tilde_psi}) to define truncated spectral transformation for the operator $L$ in the same way.
Similarly, we define $\tilde{\cal M}_l:=\left[R_{l}^{8\tilde{\tau}_0}, \min(R_{l+1}^{\frac{\tau_1}{16}}, \overline{R}_{l+1}^{\frac{\nu}{16}})\right]$.
With the same proof, we have the analogue corollary.

\begin{cor} \label{cor 2}
There exists $c_{11} > 0$, such that for any sufficiently large $l \geq 0$, for any $T > 0$ such that $[T^{\eta}, T^{101}] \subset  \tilde{\cal M}_l$, we have
$ \langle u(T)  \rangle_2 \geq c_{11} T^{1-\eta}$.
\end{cor}

\begin{proof}[Proof of Theorem \ref{main theorem 2}]
Without loss of generality, we assume that $\eta \in (0,1)$. Recall that we have given $\tau_1$, $\tau_2$ and ${\cal A}$ in (\ref{tau_1tau_2}) and (\ref{cal_A}).
Note that for any $l \geq 1$ we have
$Q_{l+1}^{\tau_0} \geq R_{l}^{\cA\tau_0} \geq R_{l}^{\tilde{\tau}_0}$ since $\cA \geq \frac{\tilde{\tau}_0}{\tau_0}$. It is clear that $Q_{l}^{\tau_0} \leq R_{l}^{\tilde{\tau}_0}$ since $Q_{l} \leq R_{l}$ and $\tau_0 \leq \tilde{\tau}_0$.
Thus for any sufficiently large $T$, we have two (nonexclusive) alternates  :
\begin{enumerate}
\item There exists $l$ such that $R_{l}^{8\tilde{\tau}_0} \leq T^{\eta} < Q_{l+1}^{8\tau_0}$. Then $T^{101} < Q_{l+1}^{808\tau_0/\eta} < \min( R_{l+1}^{\frac{\tau_{2}}{16}}, \overline{R}_{l+1}^{\frac{\nu}{16}})$ since $R_{l+1} \geq Q_{l+1}$, $\tau_2 > \frac{10^7\tau_0}{\eta}$ and $\overline{R}_{l+1}^{\frac{\nu}{16}} \geq Q_{l+1}^{\frac{\nu\cA}{16}} > Q_{l+1}^{\frac{808\tau_0}{\eta}}$ for $\cA > \frac{10^{7}\tau_0}{\eta\nu}$.

\item There exists $l$ such that $Q_{l}^{8\tau_0} \leq T^{\eta} < R_{l}^{8\tilde{\tau}_0}$. Then $T^{101} < R_{l}^{808\tilde{\tau}_0/ \eta} < \min( Q_{l+1}^{\frac{\tau_1}{16}}, \overline{Q}_{l+1}^{\frac{\nu}{16}})$ since $R_l^{\frac{808\tilde{\tau}_0}{\eta}} < Q_{l+1}^{\frac{808\tilde{\tau}_0}{\eta\cA}} < \min( Q_{l+1}^{\frac{\tau_1}{16}}, \overline{Q}_{l+1}^{\frac{\nu}{16}})$. The last inequality follows from $\cA > \frac{10^{7}\tilde{\tau}_0}{\eta} \max\left( \tau_1^{-1}, \nu^{-1} \right)$ and $\overline{Q}_{l+1} > Q_{l+1}$.
\end{enumerate}
Thus for any sufficiently large $T$, there exists $l \geq 0$ such that $[T^{\eta}, T^{101}]$ is contained in either $\cal M_l$ or $\tilde{ \cal M_l}$.
The theorem follows from Corollary \ref{cor 1} and Corollary \ref{cor 2}.
\end{proof}

\subsection{Ballistic lower bound for $p\geq 0$} \label{The general case}

Till now we already have the lower bound for $p=2$ and $p=0$(which is well-known as the $\ell^2-$conservation).
Now we consider the general $p\geq 0$. We only have to show the lower bound
$$
\lim_{t\rightarrow\infty} \frac{\lan u(t)\ran_p^2}{t^{p-\eta}}=0, \quad \forall \eta>0, \;\ \forall u(0)\in {\cal W}^p(\Z)\setminus \{0\}
$$
for any even integer $p$.
Indeed, by Cauchy inequality, we have
\begin{eqnarray*}
\langle u \rangle_{p_2} \lesssim \langle u \rangle_{p_3}^{\frac{p_1-p_2}{p_1-p_3}} \langle u \rangle_{p_1}^{\frac{p_2-p_3}{p_1-p_3}} \;\  {\rm for \ any }  \;\ p_1 \geq p_2 \geq p_3  \geq 0.
\end{eqnarray*}
\begin{itemize}
  \item If $p> 2$, we can apply the above inequality for $p_1 = p$, $p_2 = 2\lfloor p/2 \rfloor$ and $p_3 = 0$, with the lower bound of $\lan u(t) \ran_{2\lfloor p/2 \rfloor}$ and the upper bound of $\lan u(t) \ran_{0}$ known. Then we get the lower bound in Theorem \ref{main theorem 2} for $\lan u(t) \ran_{p}$.
  \item If $0 < p < 2$, we can apply the above inequality for $p_1 = 4$, $p_2 = 2$ and $p_3 = p$, with the lower bound of $\lan u(t) \ran_{2}$ and the upper bound of $\lan u(t) \ran_{4}$ known. Then we get the lower bound in Theorem \ref{main theorem 2} for $u(0)\in{\cal W}^4(\Z)\setminus\{0\}$.
\end{itemize}

For the $p^{\rm th}$ moment bound, with $p$ an even integer, we need to make the following modifications to the proof presented above :
\begin{enumerate}
\item Let $r = p /2  + 2$. We consider $G_{T^{r}}(t,E)$ and $\partial_{E}^{r-2}G_{T^{r}}(t,E)$ instead of $G_{T}(t,E)$ and $\partial_EG_{T}(t,E)$.
\item Instead of Lemma \ref{F T small}, Lemma \ref{P_F T small}, we prove  that  $\lim_{T \to \infty }\sup_{0\leq t \leq T}\norm{\partial_E^{j}F_{T^{r}}(t,E)}_{{\cal L}^2(d\mu)} = 0$ for each $0 \leq j \leq r-2$.
\item We consider $a_{k,r} := \sup_{n \in \Z} \frac{1}{\lan n+k \ran^{r-2} \lan n \ran^{r-2}}\left|\int_\R [(\partial_{E}^{r-2}{\cal K}_{n+k})(\partial_{E}^{r-2} {\cal K}_n) + (\partial_{E}^{r-2}{\cal J}_{n+k}) (\partial_{E}^{r-2}{\cal J}_n)] \psi \, dE \right|$ instead of $a_k$ in \eqref{a k}, with $\partial_{E}^{r-2} {\cal K}_n$, $\partial_{E}^{r-2} {\cal J}_n$ given in (\ref{derivative_KnJn2}).
\item When verifying the summability of $a_{k,r}$ for $|k| \in \cal M_l$, we replace the estimates $\eqref{partial_KJ_1}-\eqref{partial_KJ_4}$ by $\eqref{2partial_KJ_1}-\eqref{2partial_KJ_4}$.
\end{enumerate}

\appendix

\section{Proof of Lemma \ref{lemma conjugation} and \ref{lemma_derivative}}\label{appendix}

In this section, the integer $r$ is fixed.
The implicit constants in the notations ``$\lesssim$" and ``$\gtrsim$" are allowed to depend on $r$.

\noindent{\bf Proof of Lemma \ref{lemma conjugation}}

We follow the proof of Proposition 4.1 (1) in \cite{AFK} and construct the conjugation cocycle $B$ using  iterations of algebraic conjugation, given by the following Lemma \ref{lemma smooth alg conj}.

\begin{lemma} \label{lemma smooth alg conj}
For every $D > 0$, there exist $C_0 = C_0(D)$, $\epsilon_1 = \epsilon_1(D) > 0$ such that the following is true. Given $\epsilon \in (0, \epsilon_1)$, $J_0$ an open interval, let $\cU \subset SL(2,\C) \times \C$ be the set of all $(A, \theta)$ such that  $|A|_{J_0}< D$, $|R_{\theta}^{-1}A - id|_{J_0} < \epsilon$ and $\varrho = \min(\frac{1}{4}, |(R_{2\theta} -id)^{-1}|_{J_0}^{-1} ) > \epsilon^{\frac{1}{r+4}}$, and let $E \mapsto (A(E),\theta(E))$ be a function in $C^{\omega}(J_0, \cU)$.  There exists a real symmetric holomorphic function $F : \cU \to SL(2,\C)$ such that, with
$$H:=1+\max_{j=1,\cdots, r}\left( |\partial_E^j\theta|_{J_0}^{\frac r j}, \, |\partial_E^{j}(R^{-1}_{\theta}A)|_{J_0}^{\frac r j}\epsilon^{-\frac r j} \right),$$
$B(E) = F(A(E), \theta(E))$ satisfies $BAB^{-1} = R_{\theta'}$ for some $\theta^{'} \in C^{\omega}(J_0, \C) $, and
$$|\partial_{E}^{j}(B-id)|_{J_0} < C_0\epsilon \varrho^{-j-1} H^{\frac{j}{r}}, \quad |\partial_{E}^{j}(\theta - \theta')|_{J_0} < C_0 \epsilon H^{\frac{j}{r}}, \quad j=0,\cdots, r.$$
\end{lemma}
\proof In this proof, for each function depending on $E\in J_0$, the norm we are considering is always $|\cdot|_{J_0}$. For convenience, we will not present the subscript ``$J_0$".

Let $A^{(0)} := A$ and $\theta^{(0)} := \theta$.
Suppose that $A^{(n)}$ and $\theta^{(n)}$ are already defined. Let $v^{(n)}\in sl(2,\C)$ be small (thus unique) and satisfy that $A^{(n)} = e^{v^{(n)}}R_{\theta^{(n)}}$, with $v^{(n)}=\left( \begin{array}{cc}
x^{(n)} & y^{(n)} - 2\pi z^{(n)} \\
y^{(n)} + 2\pi z^{(n)} & -x^{(n)}
\end{array}\right)$.
We define $\left(\begin{array}{c}
       \tilde{x}^{(n)} \\
       \tilde{y}^{(n)}
     \end{array}\right):=(R_{2\theta^{(n)}}-id)^{-1}\left(\begin{array}{c}
       x^{(n)} \\
       y^{(n)}
     \end{array}\right)$ and $w^{(n)} := \left(\begin{array}{cc}
                                               - \tilde{x}^{(n)} &  -  \tilde{y}^{(n)}\\
                                                - \tilde{y}^{(n)} &  \tilde{x}^{(n)}
                                             \end{array}\right)$.
Let $A^{(n+1)} =  (e^{w^{(n)}} )^{-1}A^{(n)}e^{w^{(n)}}$ and $\theta^{(n+1)} = \theta^{(n)} + z^{(n)}$.
Then we have $$A^{(n+1)} = (e^{w^{(n)}})^{-1} e^{v^{(n)}} R_{\theta^{(n)}} e^{w^{(n)}} =( e^{w^{(n)}} )^{-1}e^{v^{(n)}}R_{\theta^{(n)}}e^{w^{(n)}}R_{-\theta^{(n+1)}}R_{\theta^{(n+1)}}.$$
Note that by our choice $R_{\theta^{(n)}}w^{(n)}R_{-\theta^{(n)}} = R_{2\theta^{(n)}}w^{(n)}$.
So $v^{(n+1)}$ is defined by $$e^{v^{(n+1)}} = (e^{w^{(n)}})^{-1} e^{v^{(n)}}R_{\theta^{(n)}}e^{w^{(n)}}R_{-\theta^{(n+1)}} = (e^{w^{(n)}})^{-1} e^{v^{(n)}} e^{R_{2\theta^{(n)}}w^{(n)}} R_{-z^{(n)}}.$$
Notice that we have
\begin{equation}\label{lem 5 term A2}
|\partial_{E}^{j}z^{(n)} |\lesssim |\partial_{E}^{j}v^{(n)}|,  \quad  \forall 0 \leq j \leq r.
\end{equation}

We consider the $k^{\rm th}$-order differential of $e^{v^{(n+1)}}$. It can be decomposed into $I_1+I_2$, where
\begin{eqnarray*}
I_1 &=&  \partial_E^{k}((e^{w^{(n)}})^{-1}) e^{v^{(n)}} e^{R_{2\theta^{(n)}}w^{(n)}} R_{-z^{(n)}} +  (e^{w^{(n)}})^{-1}\partial_E^{k} e^{v^{(n)}} e^{R_{2\theta^{(n)}}w^{(n)}} R_{-z^{(n)}} \\
& & + \, (e^{w^{(n)}})^{-1} e^{v^{(n)}} \partial_E^{k}e^{R_{2\theta^{(n)}}w^{(n)}} R_{-z^{(n)}} + (e^{w^{(n)}})^{-1} e^{v^{(n)}} e^{R_{2\theta^{(n)}}w^{(n)}} \partial_E^{k} R_{-z^{(n)}}
\end{eqnarray*}
and $I_2$ is the sum of terms involving at least two differentiated components.
By the definition of $w^{(n)}$, we have the identity
\begin{equation}\label{lem 5 term A1}
v^{(n)} + (R_{2\theta^{(n)}}-id) w^{(n)} - z^{(n)}\left(\begin{array}{cc}
                                               0 &  -  2\pi \\
                                                2\pi &  0
                                             \end{array}\right) = 0.
\end{equation}
By applying \eqref{lem 5 term A1}, we obtain
$|I_1| \lesssim (|\partial_E^{k}v^{(n)}| + |\partial_E^{k}w^{(n)}|)(|v^{(n)}| + |w^{(n)}|)$.

By the definition of $I_2$, we also have
$|I_2| \lesssim (|\partial_E^{k}v^{(n)}| + |\partial_E^{k}w^{(n)}|)(|v^{(n)}| + |w^{(n)}|)$.

We will prove inductively that for all $k \geq 0$,
\begin{eqnarray}
&&|(R_{2\theta^{(k)}}-id)^{-1}| \leq (2 - 2^{-k}) \varrho^{-1}, \;\  |\partial_{E}^{j}\theta^{(k)}| \leq (2 - 2^{-k})H^{\frac{j}{r}}, \quad j=1,\cdots,r \label{lem 5 term 2}\\
&& |\partial_E^{j}v^{(k)}| \lesssim \varrho^{k} \epsilon H^{\frac{j}{r}}, \quad  j=0,\cdots,r. \label{lem 5 term 3}\\
&& |\partial_E^{j}w^{(k)}| \lesssim \varrho^{k-j-1} \epsilon H^{\frac{j}{r}}, \quad  j=0,\cdots,r. \label{lem 5 term 32}
\end{eqnarray}
It is direct to verify \eqref{lem 5 term 2} and \eqref{lem 5 term 3} for $k = 0$ by the conditions in the lemma.
Given \eqref{lem 5 term 2}, \eqref{lem 5 term 3} for $0 \leq k \leq n$,
it is easy to show by induction that for all $0 \leq k \leq n$, $0 \leq j \leq r$
\begin{eqnarray*}
|\partial_E^{j}((R_{2\theta^{(k)}} - id)^{-1})| \lesssim \varrho^{-j-1}H^{\frac{j}{r}}
\end{eqnarray*}
Combining the above estimate with  \eqref{lem 5 term A2} and \eqref{lem 5 term A1}, we can deduce \eqref{lem 5 term 32}. In particular, we have verified \eqref{lem 5 term 32} for $k=0$.

Now assume that \eqref{lem 5 term 2}, \eqref{lem 5 term 3}, \eqref{lem 5 term 32} are true for $0 \leq k \leq n$. For each $0 \leq j \leq r$, we decompose $\partial_E^{j}(e^{v^{(n+1)}})$ into $I_1$, $I_2$ as described above. By \eqref{lem 5 term 3}, \eqref{lem 5 term 32} for $k=n$, we obtain
$$
|I_1|, \, |I_2| \lesssim \varrho^{2n-r-2} \epsilon^{2} H^{\frac{j}{r}}.
$$
By  $\epsilon \ll \varrho^{r+3}$, we have
$
|\partial_E^{j}v^{(n+1)}| < \varrho^{n+1} \epsilon H^{\frac{j}{r}}.
$
This gives us \eqref{lem 5 term 3} for $k=n+1$.
By \eqref{lem 5 term A2}, \eqref{lem 5 term 3} for $n$, and the fact that $\theta^{(n+1)} - \theta^{(n)} = z^{(n)}$, we see that for sufficiently small $\epsilon$, we have \eqref{lem 5 term 2} for $k=n+1$.
We deduce \eqref{lem 5 term 32} from \eqref{lem 5 term 2}, \eqref{lem 5 term 3}.
This completes the induction.

We define $B = \lim_{n \to \infty} (e^{w^{(n)}})^{-1} \cdots (e^{w^{(1)}})^{-1}(e^{w^{(0)}})^{-1}$.
It is clear that $BAB^{-1} = R_{\theta'}$ where $\theta' = \lim_{n \to \infty} \theta^{(n)}$.
Recall that $|\partial_E^{j}(\theta^{(n)} - \theta^{(n+1)})| \lesssim |\partial_E^{j}v^{(n+1)}|$, then there exists $C_0 = C_0(D) > 0$ such that
$$|B- id| < C_0 \varrho^{-1} \epsilon,\quad |\theta - \theta'| < C_0 \epsilon,$$
and for $j=1$, $\cdots$, $r$, we have
$$|\partial_{E}^{j}B| < C_0 \varrho^{-j-1} \epsilon H^{\frac{j}{r}}, \quad |\partial_{E}^{j}(\theta - \theta')| < C_0 \epsilon H^{\frac{j}{r}}.$$
\qed

We will construct $B$ in Lemma \ref{lemma conjugation} inductively, by following the proof of Proposition 4.1 in \cite{AFK}.
Let $A_{(0)} = \bar{A}$, $\varphi_{(0)} = \bar{\varphi}$, $\xi_{(0)} = \bar{\xi} = R_{\bar{\varphi}}^{-1}\bar{A} -id$.  Let $C_2 > 10, C_3 = C_0(10D)$ as in Lemma \ref{lemma smooth alg conj}, $C_1 \gg  r C_2C_3^{10r}$, $N = \frac{\delta h \varrho^{2r+1}}{C_1 |\bar{\alpha}|}$ and  $h_{j} = e^{-\delta \frac{j}{5N}} h$, $j \geq 0$.
For $1 \leq i \leq N$ we define
\begin{enumerate}
\item $R_{\varphi_{(i)}} = B_{(i-1)}A_{(i-1)}B_{(i-1)}^{-1}$ where $B_{(i-1)}$ is obtained by applying Lemma \ref{lemma smooth alg conj} to $(A_{(i-1)}, \varphi_{(i-1)})$;
\item $A_{(i)}(E,x) = B_{(i-1)}(E,x + \bar{\alpha})B_{(i-1)}(E,x)^{-1}R_{\varphi_{(i)}(E,x)}$;
\item $\xi_{(i)} = R_{\varphi_{(i)}}^{-1}A_{(i)} - id$.
\end{enumerate}
The following estimates are known (Claim 4.3 in \cite{AFK}), thus we omit the proof.
\begin{lemma}\label{lemma known}
If $\epsilon_0$ is sufficiently small, then
\begin{enumerate}
\item $|B_{(i)} - id|_{h_i} < C_3 \varrho^{-1} |\xi_{(i)}|_{h_i}$  for all $0 \leq i \leq N-1$  ;
\item $|R_{\varphi_{(i)}}|_{h_i} < 2D$   for all $0 \leq i \leq N$ ;
\item $|(R_{2\varphi_{(i)}} -id )^{-1}|_{h_i} < 2 \varrho^{-1}$   for all $0 \leq i \leq N$ ;
\item $|\xi_{(i)}|_{h_i} \leq  C_2^{-1}|\xi_{(i-1)}|_{h_{i-1}} \leq C_2^{-i} \epsilon_0 $   for all $1 \leq i \leq N$ .
\end{enumerate}
\end{lemma}

Denote $\epsilon_i = C_2^{-i} \epsilon_0$ for $1 \leq i \leq N$.
We will prove inductively that for $0 \leq i \leq N$,
\begin{equation}\label{lem xx term 1}
1+\max_{1 \leq j \leq r}\left(|\partial_{E}^{j}\varphi_{(i)}|^{\frac r j}_{h_i}, \, |\partial_{E}^{j}\xi_{(i)}|^{\frac r j}_{h_i}\epsilon_i^{-\frac r j}\right) \leq (2 - 2^{-i})H,
\end{equation}
\begin{equation}\label{lem xx term 2}
|\partial^j_{E}(B_{(i)} - id)|_{h_i} \leq C_3\epsilon_i \varrho^{-j-1} H^{\frac{j}{r}}, \quad j=1,\cdots, r.
\end{equation}

For $i=0$, \eqref{lem xx term 1} is given by the definition of $H$ in Lemma \ref{lemma conjugation}.  Apply Lemma \ref{lemma smooth alg conj} to $A_{(0)}$ and $\varphi_{(0)}$, we obtain \eqref{lem xx term 2}. This verifies the induction hypothesis for $i=0$.

Assume that we have \eqref{lem xx term 1}, \eqref{lem xx term 2} for $0 \leq i \leq n-1$.
For each $1\leq k \leq r$, we have
\begin{eqnarray*}
\partial_E^{k}\xi_{(n)} &=& \sum_{1 \leq i_1+i_2 \leq k} \partial_{E}^{i_1}(R_{-\varphi_{(n)}}) \partial_{E}^{k-i_1-i_2}(B_{(n-1)}(\cdot + \bar{\alpha}) B_{(n-1)}^{-1}-id)\partial_{E}^{i_2}(R_{\varphi_{(n)}})\\
 & & + \,  R_{-\varphi_{(n)}} \partial_{E}^{k}(B_{(n-1)}(\cdot + \bar{\alpha}) B_{(n-1)}^{-1}-id)R_{\varphi_{(n)}}.
\end{eqnarray*}
Differentiate both sides of the identity $B_{(n-1)}(\cdot + \bar{\alpha}) B_{(n-1)}^{-1}-id = R_{\varphi_{(n)}} \xi_{(n)}R_{-\varphi_{(n)}}$, we can see
\begin{eqnarray}
|\partial_E^{k}\xi_{(n)}|_{h_{n}} &\leq& \sum_{1 \leq i_1+i_2+j_1+j_2 \leq k} |\partial_{E}^{i_1}(R_{-\varphi_{(n)}})|_{h_n} |\partial_{E}^{j_1}(R_{\varphi_{(n)}})|_{h_n} |\partial_{E}^{k-i_1-i_2-j_1-j_2}\xi_{(n)}|_{h_n} |\partial_{E}^{j_2}(R_{-\varphi_{(n)}})|_{h_n}|\partial_{E}^{i_2}(R_{\varphi_{(n)}})|_{h_n} \nonumber \\
&& + \, |\partial_{E}^{k}(B_{(n-1)}(\cdot + \bar{\alpha}) B_{(n-1)}^{-1}-id)|_{h_n}.  \label{lem xx term 3}
\end{eqnarray}
By the analyticity, we have
\begin{eqnarray*}
|\partial_{E}^{k}(B_{(n-1)}(\cdot + \bar{\alpha}) B_{(n-1)}^{-1}-id)|_{h_n} &\leq& \sum_{j=0}^{k}|\partial_{E}^{j}(B_{(n-1)}^{-1})|_{h_{n}} |\partial_E^{k-j}(B_{(n-1)}(\cdot + \bar{\alpha}) - B_{(n-1)})|_{h_n} \\
&\lesssim& (h_{n-1} - h_{n})^{-1} \bar{\alpha} \sum_{j=0}^{k}|\partial_{E}^{j}(B_{(n-1)}^{-1})|_{h_{n-1}}|\partial_{E}^{k-j}B_{(n-1)}|_{h_{n-1}} \\
&\lesssim& \frac{N\bar{\alpha}}{\delta h} C_3^{k+1} \varrho^{-2k-1} \epsilon_{n-1} H^{\frac{k}{r}}.
\end{eqnarray*}
The last step follows from \eqref{lem xx term 2} and the inductive estimate
\begin{eqnarray*}
|\partial_E^{j}(B^{-1}_{n-1})|_{h_{n-1}} \lesssim C_3^{j}\epsilon_{n-1} \varrho^{-2j} H^{\frac{j}{r}},
\end{eqnarray*}
which we obtain from \eqref{lem xx term 2} by direct computation.
Using \eqref{lem xx term 1} and \eqref{lem xx term 3}, we can see that
\begin{eqnarray*}
|\partial_E^{k}\xi_{(n)}|_{h_n} \lesssim \sum_{j=0}^{k-1} H^{\frac{k-j}{r}} |\partial_{E}^{j}\xi_{(n)}|_{h_n} +  C_3^{k+1}\frac{N\bar{\alpha}}{\delta h} \varrho^{-2k-1} \epsilon_{n-1} H^{\frac{k}{r}}.
\end{eqnarray*}
Recall that $C_1 \gg rC_2C_3^{10r}$ and $N = \frac{\delta h \varrho^{2r+1}}{C_1 |\bar{\alpha}|}$, by simple induction, we can see that
$$
|\partial_E^{k}\xi_{(n)}|_{h_n} \leq \epsilon_{n} H^{\frac{k}{r}}, \quad 0 \leq k \leq r.
$$

By induction hypothesis \eqref{lem xx term 1} for $i=n-1$ and Lemma \ref{lemma known}(1) and (3), we can apply Lemma \ref{lemma smooth alg conj} to $A_{(n-1)}$ and $\varphi_{(n-1)}$, and show that for some $C_4 = C_4(D)$,
\begin{eqnarray*}
|\partial_{E}^{j}(\varphi_{(n)} - \varphi_{n-1})|_{h_{n}} < C_4 \epsilon_{n-1}H^{\frac{j}{r}}.
\end{eqnarray*}
Thus $|\partial_{E}^{j}\varphi_{(n)}|_{h_{n}} \leq [(2-2^{-n})H]^{\frac{j}{r}} $  for all $0 \leq j \leq r$ when $\epsilon_0$ is sufficiently small. This gives us \eqref{lem xx term 1} for $i = n$.
Combining with another application of Lemma \ref{lemma smooth alg conj} to $A_{(n)}$ and $\varphi_{(n)}$, we obtain \eqref{lem xx term 2} for $i=n$.
This completes the induction.

We define $B = B_{(N-1)} \cdots B_{(0)}$, $\tilde{\bar{A}} = A_{(N)}$, $\varphi = \varphi_{(N)}$. Then $R_{\varphi} \tilde{\bar{A}}- id = \xi_{(N)}$.
It is direct to check \eqref{lemma 6 term 101} and \eqref{lemma 6 term 102} using \eqref{lem xx term 2} and Lemma \ref{lemma known}.
Since $\tilde{\bar{A}} = R_{\varphi} + R_{\varphi} \xi_{(N)}$, we have $\mathcal{Q}(\tilde{\bar{A}}) = \mathcal{Q}(R_{\varphi} \xi_{(N)})$.
Then \eqref{lemma 6 term 103} follows from \eqref{lem xx term 1}.\qed

\

From now on, $C_0$ is a generic constant depending only on $D,r$ that varies from line to line.

\noindent
{\bf Proof of Lemma \ref{lemma_derivative} }

The following lemma is Claim 4.4 of \cite{AFK}. We present it here without proof.
\begin{lemma}\label{lemma known 2}
If $C_{5}(D) > 0$ is sufficiently large, and if $\epsilon_0$ is sufficiently small and  $|\bar{\alpha}| < C_5^{-1}\delta h\varrho^2$, then
$$
|(R_{\varphi(x +\alpha) + \varphi(x)}-id)^{-1}|_{e^{-\delta/3}h} \lesssim \varrho^{-1}.
$$
\end{lemma}

Following \cite{AFK}, we define $\tilde{\zeta}$, $\tilde{\xi}$ by letting $R_{\tilde{\zeta}} = \frac{\tilde{G} - \mathcal{Q}(\tilde{G})}{det(\tilde{G} - \mathcal{Q}(\tilde{G}))^{\frac{1}{2}}}$ and $\tilde{\xi} =R_{-\tilde{\zeta}}\tilde{G} -id$. Here $\cal{Q}$ is defined in Lemma \ref{lemma conjugation}.
Then
\begin{equation}\label{lem 6 term 2}
(\tilde{G} - \mathcal{Q}(\tilde{G}))\tilde{\xi} = (\det(\tilde{G} - \mathcal{Q}(\tilde{G}))^{\frac{1}{2}} - 1)\tilde{G} + \mathcal{Q}(\tilde{G}).
\end{equation}
After differentiating \eqref{lem 6 term 2}, we obtain that, for each $k \geq 1$,
\begin{equation}
\sum_{j=0}^{k} \partial_{E}^{k-j}(\tilde{G} - \mathcal{Q}(\tilde{G}))\, \partial_E^{j}\tilde{\xi} = \sum_{j=0}^{k}\partial_E^{j}(\det(\tilde{G} - \mathcal{Q}(\tilde{G}))^{\frac12}-1) \, \partial_E^{k-j} \tilde{G} + \partial_E^{k}\mathcal{Q}(\tilde{G}). \label{lem 6 term 3}
\end{equation}
Then, by a direct computation, we have, for any $k \geq 0$,
\begin{eqnarray}
|\partial_{E}^k \mathcal{Q}(\tilde{G})|, \; |\partial_{E}^k(\tilde{G} - \mathcal{Q}(\tilde{G}))|&\lesssim& |\partial_E^{k}\tilde{G}|,  \label{lem 6 term 4} \\
|\partial_{E}^k(\det(\tilde{G} - \mathcal{Q}(\tilde{G}))-1)|
&\lesssim& \sum_{j=0}^{k}  | \partial_{E}^{k- j} \tilde{G}| |\partial_E^{j} \mathcal{Q}(\tilde{G})|,  \label{lem 6 term 5}
\end{eqnarray}
where $|\cdot| = |\cdot|_{e^{-\frac{\delta}{2}}h}$.
By choosing $\epsilon_0$ small, we can assume that $\tilde{G} - \cal{Q}(\tilde{G})$ is invertible and $\frac{1}{2} <|\det(\tilde{G} - \cal{Q}(\tilde{G}))| < 2$. Thus
\begin{eqnarray}
|\partial_{E}^k(\det(\tilde{G} - \mathcal{Q}(\tilde{G}))^{\frac{1}{2}}-1)|
&\lesssim& \sum_l \sum_{i_j \geq 1\atop{i_1 + \cdots + i_l = k}} \prod_{j=1}^{l} |\partial_E^{i_j}\det(\tilde{G} - \mathcal{Q}(\tilde{G}))|  \nonumber \\
&\lesssim& \sum_l \sum_{i_j \geq 1, 0 \leq p_j \leq i_j \atop{i_1 + \cdots + i_l = k} }\prod_{j=1}^{l} |\partial_E^{i_j - p_j} \mathcal{Q}(\tilde{G})| | \partial_E^{p_j} \tilde{G}|.  \label{lem 6 term 6}
\end{eqnarray}
Combining with (\ref{lem 6 term 3})--(\ref{lem 6 term 6}), we obtain
\begin{eqnarray}
|\partial_{E}^{k}\tilde{\xi}| &\lesssim& \sum_{i=0}^{k-1} |\partial_E^{i}\tilde{\xi}||\partial_E^{k-i}\tilde{G}|+ |\partial_E^{k}\tilde{G}||\tilde{G}||\mathcal{Q}(\tilde{G})| + |\partial_E^{k}\cal{Q}(\tilde{G})|\nonumber \\
&& + \sum_{i=1}^{k} |\partial_{E}^{k-i}\tilde{G}| \sum_{j=1}^{i} \sum_{ 0 \leq p_m \leq j_m  \atop{\sum_{m=1}^{l} j_m = j\atop{j_m \geq 1}}} \prod_{m=1}^{l}|\partial_E^{j_m - p_m}\mathcal{Q}(\tilde{G})||\partial_E^{p_m}\tilde{G}|.\label{lem 6 term 8}
\end{eqnarray}

We have the equality $R_{-\zeta}B(\cdot + \alpha)R_{\zeta}(id + \xi) B^{-1} = R_{\tilde{\zeta}-\zeta}(id + \tilde{\xi})$.
Thus
\begin{eqnarray*}
(R_{\tilde{\zeta} - \zeta} - id)(id + \tilde{\xi}) = R_{-\zeta}(B(\cdot + \alpha) - id) R_{\zeta}(id + \xi)(id - B) B^{-1}\\  + (id + \xi)(id - B) B^{-1} + R_{-\zeta}(B(\cdot + \alpha)-id)R_{\zeta}(id + \xi).
\end{eqnarray*}
After differentiating both side, we obtain for each  $k \geq 0$   that
\begin{eqnarray}
|\partial_E^k(\tilde{\zeta} - \zeta)|
 &\lesssim& \sum_{i=1}^{k} |\partial_E^{k-i}(\tilde{\zeta} - \zeta)||\partial_E^{i}\tilde{\xi}| \nonumber \\
&&+ \, \sum_{i_1 + \cdots + i_6 = k\atop{i_j \geq 0}} |\partial_E^{i_1}R_{-\zeta}| |\partial_E^{i_3}R_{\zeta}| |\partial_E^{i_4}(id + \xi)| |\partial_E^{i_2}(B-id)| |\partial_E^{i_5}(B-id)| |\partial_E^{i_6}(B^{-1})|.  \nonumber \\
&& + \, \sum_{i_1 + i_2 + i_3 = k\atop{i_j \geq 0}} |\partial_E^{i_1}(id + \xi)| |\partial_E^{i_2}(B-id)|  |\partial_E^{i_3}(B^{-1})|  \nonumber \\
&& +  \, \sum_{i_1 + i_2+i_3 + i_4 = k\atop{i_j \geq 0}} |\partial_E^{i_1}R_{-\zeta}| |\partial_E^{i_2}(B-id)| |\partial_E^{i_3}R_{\zeta}| |\partial_E^{i_4}(id + \xi)| .  \label{lem 6 term 10}
\end{eqnarray}

We need a good upper bound for $|\partial_{E}^k \mathcal{Q}(\tilde{G})|$.
Note that for each $\theta \in \C$, any matrix $M$, we have
\begin{eqnarray*}
R_{-\theta}M = \mathcal{Q}(M R_{\theta}), \quad R_{\theta}\mathcal{Q}(M) = \mathcal{Q}(R_{\theta}M).
\end{eqnarray*}
By the commutation relation between $G$ and $\bar{A}$, we obtain $\tilde{G}(x + \bar{\alpha}) \tilde{\bar{A}}(x) = \tilde{\bar{A}}(x + \alpha) \tilde{G}(x)$. After taking the derivatives, we obtain
by the commutation relation, that for each $k \geq 1$
\begin{eqnarray*}
\sum_{i=0}^{k} \partial_{E}^{k-i}\tilde{G}(x + \bar{\alpha}) \partial_{E}^{i} \tilde{\bar{A}}(x)  = \sum_{i=0}^{k} \partial_{E}^{k-i}\tilde{\bar{A}}(x + \alpha) \partial_{E}^{i} \tilde{G}(x),
\end{eqnarray*}
and
\begin{eqnarray}
&& (R_{-\varphi(x)} - R_{\varphi(x + \alpha)})  \mathcal{Q}\left(\partial_{E}^k \tilde{G}(x)\right)  \label{lem 6 term 1} \\
&=&\mathcal{Q}\left(\partial_{E}^k \tilde{G}(x) R_{\varphi(x)} - R_{\varphi(x+\alpha)} \partial_{E}^k \tilde{G}(x)\right) \nonumber  \\
&=& \mathcal{Q}\left((\partial_{E}^k \tilde{G}(x) - \partial_E^k \tilde{G}(x + \bar{\alpha}))R_{\varphi(x)}\right)
+ \mathcal{Q}\left((\tilde{\bar{A}}(x + \alpha) - R_{\varphi(x + \alpha)}) \partial_E^k \tilde{G}(x)\right) \nonumber  \\
&& + \mathcal{Q}\left( \partial_E^k \tilde{G}(x +\bar{\alpha})(R_{\varphi(x)} - \tilde{\bar{A}}(x))\right) - \sum_{i=1}^{k} \mathcal{Q}\left((\partial_{E}^{k-i}\tilde{G}(x + \bar{\alpha}) - \partial_E^{k-i}R_{\tilde{\zeta}(x + \bar{\alpha})}) \partial_E^{i}\tilde{\bar{A}}(x)\right) \nonumber \\
&& - \sum_{i=1}^{k} \mathcal{Q}\left( \partial_E^{k-i}R_{\tilde{\zeta}(x + \bar{\alpha})} \partial_E^{i}\tilde{\bar{A}}(x)\right)
+ \sum_{i=1}^{k} \mathcal{Q}\left(\partial_E^{i}\tilde{\bar{A}}(x + \alpha) (\partial_{E}^{k-i}\tilde{G}(x ) - \partial_E^{k-i}R_{\tilde{\zeta}(x)}) \right) \nonumber  \\
&& + \sum_{i=1}^{k} \mathcal{Q}\left( \partial_E^{i}\tilde{\bar{A}}(x + \alpha) \partial_E^{k-i}R_{\tilde{\zeta}(x)} \right). \nonumber
\end{eqnarray}

 To simplify notations, let $C_1 \gg r$, $N = \frac{\delta h \varrho^{2r+1}}{C_1 |\bar{\alpha}|}$, $g_n =  e^{-\frac{\delta}{2} - \frac{n\delta}{2(r+1)N}}h$, and
  let $|\cdot|_{n} := |\cdot|_{g_n}$ in the following computations.
 For $j=0,\cdots,r$, we denote $P^{(j)}(x):= \mathcal{Q}\left(\partial^{j}_{E}\tilde{G}(x)\right)$ and
 $$P^{(j)}_1(x) := P^{(j)}(x + \bar{\alpha}) - P^{(j)}(x),\quad P^{(j)}_2(x) := \mathcal{Q}\left(\partial^{j}_{E}\tilde{G}(x)R_{\varphi(x)} - R_{\varphi(x+\alpha)} \partial^{j}_{E}\tilde{G}(x)\right).$$
 By hypothesis in Lemma \ref{lemma conjugation}, Lemma \ref{lemma_derivative} and \eqref{lemma 6 term 101}, \eqref{lemma 6 term 103} (Note that the conditions and conclusions of Lemma \ref{lemma conjugation} holds for Lemma \ref{lemma_derivative} ),  we see that
\begin{eqnarray}
|\partial_E^{j}\tilde{G}|,\,  |\partial_E^j \tilde{\bar{A}}| &\lesssim& H^{\frac{j}{r}}, \quad \forall  j=0,\cdots, r,  \label{lem 7 term 100} \\
|\partial_E^{j}\mathcal{Q}(\tilde{\bar{A}})| &<& C_0 \epsilon_0 e^{-\frac{\delta h \varrho^{2r+1}}{C_0 \|\bar{\alpha}\|}} H^{\frac{j}{r}}, \quad \forall j=1,\cdots, r. \label{lem 7 term 101}
\end{eqnarray}

Assume that for $0 \leq k \leq r$, we have
\begin{equation}\label{lem 7 term 103}
|\partial_{E}^{j}\tilde{\xi}|_{jN}<  C_0 e^{-\frac{\delta h \varrho^{2r+1}}{C_0|\bar{\alpha}|}}H^{\frac{j}{r}},
\quad |\partial_{E}^{j}(\tilde{\zeta} - \zeta)|_{jN}
<C_0H^{\frac{j}{r}} \epsilon_0, \quad 0 \leq j \leq k-1.
\end{equation}
Notice that this hypothesis is empty for $k=0$.
For $kN \leq n \leq (k+1)N-1$,
by Lemma \ref{lemma known 2} and the analyticity, we obtain
\begin{eqnarray} \label{lem 7 term 104}
|P^{(j)}|_{n+1}< C_0 \varrho^{-1} |P^{(j)}_2|_{n+1},\quad |P^{(j)}_1|_{n+1} < \frac{C_0|\bar{\alpha}|}{g_{n} - g_{n+1}}|P^{(j)}|_{n}.
\end{eqnarray}
By \eqref{lem 6 term 1}--\eqref{lem 7 term 103}, we have
\begin{eqnarray*}
|P^{(k)}_2|_{n+1} &\lesssim& |P^{(k)}_1|_{n+1} + |\partial_{E}^k\tilde{G}|_{n+1} |R_{\varphi} - \tilde{\bar{A}}|_{n+1} \\
&& + \sum_{i=1}^{k}\left( |\partial_E^{k-i}(\tilde{G} - R_{\tilde{\zeta}})|_{n+1} |\partial_E^{i}\tilde{\bar{A}}|_{n+1} +  |\partial_E^{k-i}\tilde{\zeta}|_{n+1} |\partial_E^{i}\mathcal{Q}(\tilde{\bar{A}})|_{n+1} \right)  \\
&\lesssim& |P^{(k)}_1|_{n+1} + C_0 e^{-\frac{\delta h \varrho^{2r+1}}{C_0|\bar{\alpha}|}}H^{\frac{k}{r}}.
\end{eqnarray*}
Then by \eqref{lem 7 term 104}, we obtain $|P^{(k)}|_{n+1} \lesssim \frac{C_0 (r+1)N|\bar{\alpha}|}{\delta h}\varrho^{-1}|P^{(k)}|_n + C_0 e^{-\frac{\delta h \varrho^{2r+1}}{C_0|\bar{\alpha}|}}H^{\frac{k}{r}}$.
Using \eqref{lem 7 term 100}, \eqref{lem 6 term 4} and the facts that $C_1 \gg r$,  $N = \frac{\delta h \varrho^{2r+1}}{C_1 |\bar{\alpha}|}$, we get
\begin{equation}\label{partial_3_Q}
|\mathcal{Q}(\partial_{E}^k\tilde{G})|_{(k+1)N} = |P^{(k)}|_{(k+1)N}  < C_0 e^{-\frac{\delta h \varrho^{2r+1}}{C_0|\bar{\alpha}|}}H^{\frac{k}{r}}.
\end{equation}
By \eqref{lem 6 term 8} and \eqref{lem 6 term 10}, combining with \eqref{lemma 6 term 101}, \eqref{lem 7 term 100}, \eqref{lem 6 term 4}, \eqref{partial_3_Q}, we have
$$|\partial_{E}^{k}\tilde{\xi}|_{(k+1)N}<  C_0 e^{-\frac{\delta h \varrho^{2r+1}}{C_0|\bar{\alpha}|}}H^{\frac{k}{r}},
\quad |\partial_{E}^{k}(\tilde{\zeta} - \zeta)|_{(k+1)N}
<C_0H^{\frac{k}{r}} \epsilon_0.$$
This completes the induction and thus completes the proof.
\qed

\

\

{\footnotesize \noindent Zhiyuan Zhang\\
Institut de Math\'{e}matique de Jussieu---Paris Rive Gauche, B\^{a}timent Sophie Germain, Bureau 652\\
75205 PARIS CEDEX 13, FRANCE\\
Email address: zzzhangzhiyuan@gmail.com

\

\noindent Zhiyan Zhao\\
Universit\'{e} de Nice-Sophia Antipolis(ANR contracted), Laboratoire J.A. Dieudonn\'{e}, Bureau 4.819\\
06108 NICE CEDEX 02, FRANCE\\
Email address: zyqiao1985@gmail.com\\
\thanks{Research of Z.Zhao was supported by ANR grant
``ANR-14-CE34-0002-01'' for the project "Dynamics and CR geometry".}

}

\end{document}